\documentclass[12pt,a4paper]{article}

\usepackage{amsmath,amsthm,amsfonts,amssymb}
\usepackage{hyperref}            

\usepackage[shortlabels]{enumitem}

\newcommand{\Frac}[2]{\genfrac{}{}{}{0}{#1}{#2}}

\def\eps{\varepsilon}
\def\R{\mathbb{R}}
\def\N{\mathbb{N}}
\def\Z{\mathbb{Z}}
\def\ad{\mathrm{ad}}

\def\diam{\mathop{\mathrm{diam}}}

\def\ss{\mathcal{S}}
\newcommand{\hh}{{\mathcal H}}

\newtheorem{theorem}{Theorem}[section]
\newtheorem{lemma}[theorem]{Lemma}
\newtheorem{proposition}[theorem]{Proposition}
\newtheorem{corollary}[theorem]{Corollary}
\newtheorem*{theorem*}{Theorem}
\newtheorem*{lemma*}{Lemma}

\theoremstyle{definition}

\newtheorem{definition}{Definition}[section]

\theoremstyle{remark}

\newtheorem{example}{Example}[section]
\newtheorem{remark}{Remark}[section]

\begin{document}

\title{Control of
  Nonholonomic Systems and Sub-Riemannian Geometry}

\author{Fr\'{e}d\'{e}ric \textsc{Jean}%
\thanks{ENSTA ParisTech, UMA, 828 Boulevard des Mar\'{e}chaux
91762 Palaiseau, France  and
  Team  GECO, INRIA Saclay --
\^{I}le-de-France,  {\tt frederic.jean@ensta-paristech.fr}}
\thanks{This work was supported
by the ANR project {\it GCM}, program ``Blanche'',
project number NT09\_504490, and by the Commission
of the European Communities under the 7th Framework Programme Marie
Curie Initial Training Network (FP7-PEOPLE-2010-ITN), project SADCO,
contract number 264735.}}

\date{Lectures given at the
CIMPA School ``G\'{e}om\'{e}trie sous-riemannienne'', Beirut, Lebanon, 2012}

\maketitle


\tableofcontents

\newpage

Nonholonomic systems are control systems which depend linearly on the
control. Their underlying geometry is the sub-Riemannian geometry,
which plays for these systems the same role as Euclidean geometry does
for linear systems. In particular the usual notions of approximations
at the first order, that are essential for control purposes, have to
be defined in terms of this geometry. The aim of these notes is to
present these notions of approximation and their link with the metric
tangent structure in sub-Riemannian geometry.

The notes are organized as follows. In Section~\ref{se:geomnhs} we
introduce the basic definitions on nonholonomic systems and
sub-Riemannian geometry. Section~\ref{se:contro} is devoted to the
study of the controllability of nonholonomic systems, and to the
topological properties of sub-Riemannian
distances. Section~\ref{se:1storder} provides a detailed exposition of
the notions of first-order approximation, including nonholonomic
orders, privileged coordinates, nilpotent approximations, and distance
estimates such as the Ball-Box Theorem. We then see in
Section~\ref{se:tgtstruct} how these notions allow us to describe the
tangent structure to a Carnot-Carath\'{e}odory space (the metric space
defined by a sub-Riemannian distance). Finally, we present in the
appendix some results on flows in connection with the Hausdorff
formula (Section~\ref{se:flowsvf}), and some proofs on privileged
coordinates (Section~\ref{se:proofpriv}).

\section{Geometry of nonholonomic systems}
\label{se:geomnhs}

Throughout these notes we work in a  smooth $n$-dimensional manifold
$M$. However most of our considerations are local, so $M$ can also be
thought of as an open subset of $\R^n$.

\subsection{Nonholonomic systems}
\label{se:snh}

A  \emph{nonholonomic system} on $M$ is a
control system which is of the form
\begin{equation}
\tag{$\Sigma$}
\label{eq:snh}
\dot{q}=\sum_{i=1}^{m}u_iX_i(q),\  q\in M,\ \
u=(u_1,\dots,u_m)\in\mathbb{R}^m,
\end{equation}
where $X_1,\dots,X_m$ are $C^\infty$ vector fields on $M$. To give a
meaning to such a control system, we have to define what are its
solutions, that is, its trajectories.

\begin{definition}
A \emph{trajectory
  of}~\eqref{eq:snh} is a path $\gamma: [0,T] \to M$ for which there exists a function $u(\cdot) \in
L^1([0,T],\R^m)$
such that $\gamma$ is a solution
of the ordinary differential equation:
$$
\dot{q}(t)=\sum_{i=1}^{m}u_i (t) X_i (q(t)), \quad \hbox{for a.e.\ } t
\in [0,T].
$$
Such a function $u(\cdot)$ is called a \emph{control} associated with
$\gamma$.
\end{definition}
Equivalently, a trajectory is an absolutely continuous path $\gamma$
on $M$ such that
$\dot \gamma (t) \in  \Delta (\gamma (t))$ for almost every $t\in
[0,T]$, where we have set, for every $q\in M$,
\begin{equation}
  \label{eq:deltap}
  \Delta (q) = \mathrm{span}
\left\{X_1(q),\dots,X_m(q)\right\} \subset T_qM.
\end{equation}
Note that the rank of the vector spaces $\Delta(q)$ is a function of
$q$, which may be non constant. If it is constant, $\Delta$ defines a
distribution on $M$, that is, a subbundle of $TM$.

\begin{example}[unicycle]
\label{ex:unicycle}
The most typical example of nonholonomic system is the simplified
kinematic model of a unicycle. In this model, a configuration
$q=(x,y,\theta)$ of the unicycle is described by the planar
coordinates $(x,y)$ of the contact point of the wheel with the ground,
and by the angle $\theta$ of orientation of the wheel with respect to
the $x$-axis. The space of configurations is then the manifold $\R^2
\times S^1$.

The wheel is subject to the constraint of rolling without slipping,
which writes as $\dot x \sin \theta - \dot y \cos \theta =0$, or,
equivalently as $\dot q \in \ker \omega(q)$, where $\omega$ is the
one-form  $\sin \theta dx - \cos \theta dy$. Hence the set $\Delta$
of~\eqref{eq:deltap} is $\ker \omega$.

Choosing as controls the tangential velocity $u_1$ and the angular velocity $u_2$, we obtain the  nonholonomic system $\dot q = u_1 X_1(q) + u_2 X_2(q)$ on $\R^2 \times S^1$, where $X_1=\cos \theta \partial_x + \sin \theta \partial_y$, and $X_2= \partial_\theta$.
\end{example}

Let us mention here  a few properties of the trajectories of~\eqref{eq:snh}
(for more details, see~\cite{rif12}).
\begin{itemize}
\item
Fix $p \in M$ and $T>0$. For every control $u(\cdot) \in
L^1([0,T],\R^m)$, there exists $\tau \in (0,T]$ such that the Cauchy
problem
\begin{equation}
  \label{eq:cauchpb}
  \left\{
    \begin{array}[c]{l}
      \dot{q}(t)=\sum_{i=1}^{m}u_i (t) X_i (q(t)) \quad \hbox{for a.e.\ }t \in
      [0,\tau], \\
q(0) =p,
    \end{array}
\right.
\end{equation}
has a unique solution denoted by $\gamma_u$ or $\gamma(\cdot ; p,
u)$. It is called the \emph{trajectory issued from $p$ associated with $u$}.

  \item If the rank of $X_1,\dots,X_m$ is constant and equal to $m$ on
    $M$, every trajectory is associated with a unique
    control. Otherwise different controls can be associated with the
    same trajectory. In this case it will sometimes be useful to
    consider among these controls  only the ones which minimize the
    $L^1$ norm $\int \|u(t)\|dt$. By convexity, this defines a unique
    control with which the trajectory is associated.

\item Any time-reparameterization of a trajectory is still a
  trajectory: if $\gamma: [0,T] \to M$ is a trajectory associated
  with a control $u$, and $\alpha:
  [0,S] \to [0,T]$ is a $C^1$-diffeomorphism, then $\gamma \circ
  \alpha : [0,S] \to M$ is a trajectory associated with the control
  $\alpha'(s) u \left( \alpha(s)\right)$. In particular, one can
  reverse time along $\gamma$: the resulting path  $\gamma(T-s)$,
  $s\in [0,T]$, is a   trajectory associated with the control  $-u(T-s)$.

\end{itemize}

In this context, the first question is the one of the
controllability: can we join any two points by a trajectory? This suggests to introduce the following definition.
\begin{definition}
\label{de:Ap}
The \emph{attainable set from $p \in M$} is
defined to be the set $\mathcal{A}_p$ of points
attained by a trajectory of~\eqref{eq:snh} issued from $p$.
\end{definition}

The question above then becomes: is the attainable set from any point equal to the
whole manifold $M$? We will answer this question
in Section~\ref{se:contro}.

In the case where the answer is positive, next issues are notably the
motion planning (i.e.\ find a trajectory joining two given points) and
the stabilization (i.e.\ design the control as a function $u(q)$ of the state
 in such a way that the resulting differential equation is
stable).
The usual way to deal with these problems is to use a first-order
approximation  of
the system. The underlying idea is the following. Consider a nonlinear
control system
in $\R^n$,
$$
\dot x = f(x,u), \quad  x \in \R^n, \ u \in \R^m,
$$
and a pair $(\bar x,\bar u) \in \R^{m+n}$ such that $f(\bar x,\bar
u)=0$. The linearized system around this equilibrium pair is defined to be
the linear control system:
$$
\dot {\delta x} = \frac{\partial f}{\partial x}(\bar x,\bar u) \delta
x + \frac{\partial f}{\partial u}(\bar x,\bar u) \delta u, \quad
\delta x \in \R^n, \ \delta u \in \R^m.
$$
If this linearized system is controllable, so is the nonlinear one
near $\bar x$. In this case the solutions to the motion planning
and stabilization problems for the linearized system may be used to
construct solutions of the corresponding problems for the nonlinear
system (see for instance~\cite{kha95}). Thus, locally, the study of
the control system amounts to the one of the linearized system.

Does this strategy apply to nonholonomic systems? Consider a
nonholonomic system~\eqref{eq:snh} defined on an open subset
$M$ of $\R^n$. For every $\bar q \in M$, the pair $(\bar q, 0)$ is an
equilibrium pair and the corresponding linearized system is
$$
\dot {\delta x} = \sum_{i=1}^{m} \delta u_i X_i(\bar q), \quad
\delta x \in \R^n, \ \delta u \in \R^m.
$$
For this linearized system, the attainable set from a point $\delta q$
is obviously the affine subset
$$
\delta q + \Delta (\bar q) = \delta q + \mathrm{span}
\left\{X_1(\bar q),\dots,X_m(\bar q)\right\}.
$$
Thus, except in the very special case where $\mathrm{rank} \Delta (\bar q)
=n$, the linearized system is not controllable and the strategy above
does not apply, whereas nonholonomic systems may be
controllable (and generically they are), as we will
see Section~\ref{se:contro}.

This may be explained as follows. The linearization is a first-order
approximation with respect to a Euclidean (or a Riemannian)
distance. However for nonholonomic systems the underlying distance
is a sub-Riemannian one and it behaves very differently from a Euclidean
one. Thus, the local behaviour should be
understood through the study of a first-order approximation with
respect to this sub-Riemannian distance, not through the linearized
system.

We will introduce now the sub-Riemannian distances. In
Section~\ref{se:1storder}  we will see how to construct first-order
approximations with
respect to this kind of distances, and how to use them for motion
planning for instance.


\subsection{Sub-Riemannian distance}
\label{se:srdistance}

A nonholonomic system induces a distance on $M$ in the following way.
We first define the {\em sub-Riemannian metric} associated with~\eqref{eq:snh}
to be the function $g : TM \to \overline{\R}$ given by
\begin{equation}
  \label{eq:srmetric}
  g(q,v) = \inf \left\{ u_1^2+ \cdots + u_m^2 \ : \
    \sum_{i=1}^{m}u_iX_i(q)=v \right\},
\end{equation}
for $q\in M$ and $v \in T_qM$, where we adopt the convention that
$\inf \emptyset = + \infty$.
This function  $g$ is  smooth and satisfies:
\begin{itemize}
\item $g(q,v)= +\infty$ if $v \not\in \Delta (q)$,
\item $g$ restricted to $\Delta (q)$ is a positive definite
  quadratic form.
\end{itemize}

Such a metric allows to define a distance in the same way as in
Riemannian geometry.

\begin{definition}
\label{de:srdist}
The \emph{length} of an absolutely continuous
path $\gamma (t)$, $t \in [0,T]$, is
$$
{\mathrm{length}} (\gamma) = \int_0^T \sqrt{g \left(\gamma (t), \dot \gamma
(t)\right)} dt,
$$
and the {\em sub-Riemannian distance} on $M$  associated
with the nonholonomic
system~\eqref{eq:snh} is defined by
$$
d(p,q) = \inf {\mathrm{length}} (\gamma),
$$
where the infimum is taken over all  absolutely continuous paths
$\gamma$
joining $p$ to $q$.
\end{definition}
Note that only trajectories of~\eqref{eq:snh} may have a finite
length. In particular, if no trajectory
joins $p$ to $q$, then $d(p,q)=+\infty$. We will see below in
Corollary~\ref{le:d=dist} that, under an extra assumption on the
nonholonomic system, $d$ is actually a
distance function.

\begin{remark}
When $\gamma$ is a trajectory, its length is also equal to
$$
\min \int_0^T \|u(t)\| dt,
$$
the minimum being taken over all control $u(\cdot)$ associated with
$\gamma$. As already noticed, this minimum is attained at a unique
control which could be defined  as \emph{the} control associated with
$\gamma$.
\end{remark}

An important feature of the length of a path is that it is
independent of the parametrization of the path. As a consequence, the
sub-Riemannian distance $d(p,q)$ may also be understood as the
minimal time needed for the nonholonomic system to go from $p$ to $q$
with bounded controls, that is,
\begin{equation}
  \label{eq:tpsmin}
 d(p,q) = \inf \left\{
T\geq 0 \ : \
\begin{array}[c]{c}
  \exists \hbox{ a trajectory } \gamma_u: [0,T] \to M \hbox{ s.t.}\\
\gamma_u(0) = p, \ \gamma_u(T)=q,\\
\hbox{ and } \|u(t)\| \leq 1 \hbox{ for a.e.\ } t \in [0,T]
\end{array}
\right\} .
\end{equation}
This formulation justifies the assertion made in Section~\ref{se:snh}:
for nonholonomic systems, first-order approximations with respect to
the time should be understood as first-order approximations with
respect to the sub-Riemannian distance.

Another consequence of~\eqref{eq:tpsmin} is that $d(p,q)$ is the
solution of a time-optimal control problem. It then results from
standard existence theorems (see for instance~\cite{Lee1967} or~\cite{rif12}) that,
when $p$ and $q$
are sufficiently close  and $d(p,q) <\infty$, there exists  a
trajectory $\gamma$ joining $p$ to $q$ such that
$$
{\mathrm{length}} (\gamma) = d(p,q) .
$$
Such a trajectory is called a \emph{minimizing
trajectory}.

\begin{remark}
\label{re:velocity1}
Any reparameterization of a minimizing trajectory is also
minimizing. Therefore any pair of close enough points can be joined by
a minimizing trajectory \emph{of velocity one}, that is, a trajectory
$\gamma$ such that $g(\gamma(t),\dot \gamma(t)) = 1$ for a.e.\ $t$. As
a consequence, there exists a control $u(\cdot)$
associated with $\gamma$ such that $\|u(t)\|=1$ a.e. Every sub-arc of
such a trajectory $\gamma$ is also clearly minimizing, hence the equality
$d(p,\gamma(t)) = t$  holds along $\gamma$.
\end{remark}

\subsection{Sub-Riemannian manifolds}

The distance $d$ defined in Section~\ref{se:srdistance} does not always
meet the classical notion of
sub-Riemannian distance arising from a sub-Riemannian manifold. Let us
recall the latter definition.

A {\em sub-Riemannian manifold} $(M,D,g_R)$ is a smooth manifold
$M$ endowed with  a {\em sub-Riemannian
  structure}  $(D,g_R)$, where:
\begin{itemize}
\item $D$ is a distribution on $M$, that is a subbundle of $TM$;
\item  $g_R$ is a Riemannian metric on $D$, that is a smooth
  function $ g_R : D \rightarrow \R$ whose restrictions  to
  $D(q)$ are positive definite quadratic forms.
\end{itemize}
The {\em sub-Riemannian metric} associated with $(D,g_R)$ is the
function $g_{SR} : TM \to \overline{\R}$ given by
\begin{equation}
  \label{eq:srmetric2}
  g_{SR}(q,v) =
\left\{
  \begin{array}[c]{ll}
    g_R(q,v) & \hbox{if } v\in D(q), \\
+ \infty & \hbox{otherwise}.
  \end{array} \right.
\end{equation}
The  {\em sub-Riemannian distance} $d_{SR}$ on $M$ is then defined
from the metric $g_{SR}$ as $d$ is defined from the metric $g$ in Section~\ref{se:srdistance}.

What is the difference between the two constructions, that is, between
the definitions~\eqref{eq:srmetric} and~\eqref{eq:srmetric2} of a
sub-Riemannian metric?

Consider a {\em sub-Riemannian
  structure}  $(D,g_R)$. Locally, on some open subset $U$, there exist
vector fields $X_1, \dots, X_m$ whose values at each point $q \in U$
form an orthonormal
basis of $D(q)$ for the quadratic form $g_R$; the metric $g_{SR}$
associated with $(D,g_R)$ then coincides with the metric $g$ associated
with $X_1, \dots, X_m$.
Thus, locally, there is a one-to-one correspondence between
sub-Riemannian structures and nonholonomic systems for which the rank
of $\Delta (q) = \mathrm{span}
\left\{X_1(q),\dots,X_m(q)\right\}$  is constant.

However  this correspondence does not hold
globally since, for topological reasons, a distribution of rank
$m$ may not always be generated by $m$ vector fields on the whole $M$.
Conversely, the vector fields $X_1,\dots,X_m$ of a nonholonomic system
do not always generate a linear space $\Delta(q)$ of constant rank
equal to $m$. It
may even be impossible, again for topological reasons (for instance,
on an even dimensional sphere).

A way to conciliate both notions is to generalize the definition of
sub-Riemannian structure.

\begin{definition}
A {\em generalized sub-Riemannian structure} on $M$ is a triple $(E, \sigma,
g_R)$ where
\begin{itemize}
  \item $E$ is a vector bundle over $M$;
  \item $\sigma : E \rightarrow TM$ is a morphism of vector
  bundles;
  \item $g_R$ is a Riemannian metric on $E$.
\end{itemize}
\end{definition}

With a generalized sub-Riemannian structure a metric is associated
which is defined by
$$
g_{SR}(q,v) = \inf \{g(q,u)  \ : \ u \in E(q), \ \sigma(u) = v \},
\quad \hbox{for } q \in M, \ v \in T_qM.
$$
The generalized sub-Riemannian
distance $d_{SR}$ on $M$ is then defined
from this metric $g_{SR}$ as $d$ is defined from the metric $g$.

This definition of sub-Riemannian distance actually contains the two
notions of
distance we have introduced before.
\begin{itemize}
  \item Take $E = M \times \R^m$, $\sigma : E \rightarrow TM$,
  $\sigma (q,u)= \sum_{i=1}^m u_i X_i (q)$ and $g_R$ the Euclidean
  metric on $\R^m$. The resulting generalized sub-Riemannian distance
  is the distance associated with the nonholonomic
  system~\eqref{eq:snh}.

  \item Take $E=D$, where $D$ is a distribution on $M$, $\sigma : D
    \hookrightarrow TM$ the inclusion,
  and $g_R$ a Riemannian metric on $D$. We recover the
  distance   associated with the sub-Riemannian structure $(D,g_R)$.
\end{itemize}

Locally, a generalized sub-Riemannian structure
can always be defined  by a single finite family $X_1, \dots, X_m$ of
vector fields, and so by a nonholonomic system (without rank
condition). It actually appears that this is also true globally
(see~\cite{Agrachev2011a}, or~\cite{Drager2012} for the fact that a
submodule of $TM$ is finitely generated): any generalized
sub-Riemannian distance may be associated with a nonholonomic system.

In these notes, we will always consider a sub-Riemannian distance $d$
associated with a nonholonomic system. However, as we just noticed, all
the results actually hold
for a generalized sub-Riemannian distance.


\section{Controllability}
\label{se:contro}

Consider a nonholonomic system
\begin{equation}
\tag{$\Sigma$}
\dot{q}=\sum_{i=1}^{m}u_iX_i(q),
\end{equation}
on a smooth $n$-dimensional manifold $M$. This section is concerned with  the
question of \emph{controllability}: is the attainable set
$\mathcal{A}_p$ from any point
$p$ equal to the whole manifold $M$? We will see next the implications
on the sub-Riemannian distance $d$ and on the topology of the metric
space $(M,d)$.

\subsection{The Chow-Rashevsky Theorem}
\label{se:chow}
The controllability of~\eqref{eq:snh} is mainly characterized by the
properties of the Lie algebra generated by $X_1, \dots, X_m$. We first
introduce notions and definitions on this subject.

Let $VF(M)$ denote the set of smooth vector fields on $M$. We define
$\Delta^1$ to be the linear subspace of $VF(M)$ generated by
$X_1, \dots, X_m$,
$$
\Delta^1 = \mathrm{span} \{ X_1, \dots, X_m\}.
$$
For $s\geq 1$, define  $\Delta^{s+1}= \Delta^s + [\Delta^1,\Delta^s]$,
where we have set $[\Delta^1,\Delta^s]=
\mathrm{span} \{ [X,Y] \ : \ X \in \Delta^1, \ Y \in \Delta^s \}$. The
\emph{Lie algebra generated by} $X_1, \dots, X_m$ is defined to be
$\mathit{Lie} (X_1, \dots, X_m)= \bigcup_{s\geq 1} \Delta^s$. Due to the
Jacobi identity, $\mathit{Lie} (X_1, \dots, X_m)$ is the smallest
linear subspace
of $VF(M)$ which both contains $X_1,\dots, X_m$ and is invariant by Lie
brackets.

Let us denote by $I=i_1 \cdots i_k$  a multi-index of $\{1,\dots,m\}$, and
by $|I|=k$ the
length of $I$. We set
$$
X_I=[X_{i_1},[ \dots, [X_{i_{k-1}},X_{i_k}] \dots].
$$
With these notations, $\Delta^s = \mathrm{span} \{ X_I \ : \ |I| \leq
s \}$.

For $q\in M$, we set  $\mathit{Lie} (X_1, \dots, X_m)(q) =  \{ X(q)
: \, X \in \mathit{Lie} (X_1, \dots, X_m)  \}$, and, for $s\geq 1$,
$\Delta^s (q) =  \{ X(q) \ : \ X \in \Delta^s  \}$. By definition
these sets are linear
subspaces of $T_qM$.

\begin{definition}
We say that~\eqref{eq:snh} (or the vector fields $X_1,\dots,X_m$)
satisfies {\em Chow's
  Condition} if
$$
\mathit{Lie} (X_1, \dots, X_m)(q)=T_qM, \qquad \forall q
\in M.
$$
\end{definition}
\noindent Equivalently, for any $q \in M$, there exists an integer
$r=r(q)$ such that $\dim \Delta^r (q)=n$.

This property is also known as the Lie algebra
rank condition (LARC), and as the H\"{o}rmander
condition (in the context of PDE).

\begin{lemma}
\label{le:Apopen}
If \eqref{eq:snh}
satisfies Chow's Condition, then for every $p\in M$, the set
$\mathcal{A}_p$ is a
neighbourhood of $p$.
\end{lemma}

\begin{proof}
We  work in  a  small neighbourhood $U
\subset M$ of $p$ that we identify with
a neighbourhood of $0$ in $\R^n$ .

Let $\phi_t^i=\exp(tX_i)$ be the flow of the vector field $X_i$,
$i=1,\dots,m$. Every curve
$t \mapsto \phi^i_t(q)$ is a trajectory of~\eqref{eq:snh} and we have
$$
\phi_t^i = \mathrm{id} + t X_i + o(t).
$$
For every multi-index $I$ of $\{1,\dots,m\}$,  we define the local
diffeomorphisms $\phi^I_t$ on $U$ by  induction on the length $|I|$ of
$I$:  if $I=iJ$, then
$$
\phi_t^{iJ}=[\phi_t^i,\phi_t^J]
:=\phi_{-t}^J \circ \phi_{-t}^i
\circ \phi_t^J \circ\phi_t^i.
$$
By construction, $\phi^I_t (q)$
is the endpoint of a trajectory of \eqref{eq:snh} issued from
$q$. Moreover, on a neighbourhood of $p$ there holds
\begin{equation}
\label{eq:phiI}
\phi^I_t =\mathrm{id} + t^{|I|} X_{I} + o(t^{|I|}).
\end{equation}
We postpone the proof of this formula to the Appendix
(Proposition~\ref{le:quasich}).

To obtain a diffeomorphism whose derivative with respect to the
time is exactly $X_I$, we set
$$
\psi^I_t = \left\{
\begin{array}[c]{ll}
\phi^I_{t^{1/|I|}} & \hbox{if } t \geq 0, \\[2mm]
\phi^I_{-|t|^{1/|I|}} & \hbox{if } t < 0 \hbox{ and } |I| \hbox{
is
  odd}, \\[2mm]
{[} \phi^J_{|t|^{1/|I|}},\phi^i_{|t|^{1/|I|}} {]}  & \hbox{if } t
< 0 \hbox{
  and } |I| \hbox{ is even},
\end{array}
\right.
$$
where $I=iJ$.
Thus
\begin{eqnarray}
\label{eq:psii}
\psi^I_t =\mathrm{id} + t X_{I} + o(t),
\end{eqnarray}
and  $\psi^I_t (q)$ is the endpoint of a trajectory of \eqref{eq:snh} issued from $q$.

Let us choose now commutators $X_{I_1}, \dots,X_{I_n}$ whose
values at $p$ span $T_pM$. This is possible thanks to Chow's
Condition. We introduce the map $\varphi$ defined on a
small
neighbourhood $\Omega$ of $0$ in $\R^n$ by
$$
\varphi (t_1, \dots , t_n)= \psi^{I_n}_{t_n} \circ \cdots \circ
\psi^{I_1}_{t_1} (p) \in M.
$$
We conclude from (\ref{eq:psii}) that this map is $C^1$ near $0$ and has an invertible
derivative at $0$, which implies that it is a
local $C^1$-diffeomorphism. Therefore $\varphi(\Omega)$ contains a
neighbourhood of $p$.

Now, for every $t \in \Omega$, $\varphi(t)$  is  the endpoint of a
concatenation of trajectories of \eqref{eq:snh}, the first one being issued
from $p$. It
is then the endpoint of a trajectory starting from $p$.  Therefore
$\varphi(\Omega) \subset \mathcal{A}_p$, which implies that
$\mathcal{A}_p$ is a neighbourhood of
$p$. 
\end{proof}

\begin{theorem}[Chow-Rashevsky's theorem]
\label{thmCR}
If $M$ is connected and if \eqref{eq:snh} satisfies Chow's Condition,
then any two
points of $M$ can be joined by a trajectory of \eqref{eq:snh}.
\end{theorem}

\begin{proof}
Let $p\in M$.  If $q\in \mathcal{A}_p$, then $p \in \mathcal{A}_q$. As
a consequence, $\mathcal{A}_p=\mathcal{A}_q$ for any $q\in M$ and the
lemma above implies that $\mathcal{A}_p$ is an open set. Hence the
manifold $M$ is covered by the union of the sets $\mathcal{A}_p$ that
are pairwise disjointed. Since $M$ is connected, there is only
one such open set.
\end{proof}

\begin{remark}
This  theorem appears also as a consequence of the Orbit Theorem
(Sussmann, Stefan~\cite{ste74,sus73}):
each set $\mathcal{A}_p$ is a connected immersed submanifold of $M$ and,
at each point $q \in \mathcal{A}_p$, $\mathit{Lie}(X_1, \dots, X_m)(q) \subset
T_q\mathcal{A}_p$.
Moreover, when the rank of the Lie algebra is constant on $M$, both
spaces are
equal, i.e.\ $\mathit{Lie}(X_1, \dots, X_m)(q)= T_q\mathcal{A}_p$.

\noindent Thus, when the Lie algebra generated by $X_1, \dots, X_m$
has constant rank, Chow's Condition is not restrictive: it is indeed
satisfied on each
$\mathcal{A}_p$ by the restriction  of
the vector fields $X_1,\dots,X_m$ to the manifold $\mathcal{A}_p$.
\end{remark}

\begin{remark}
The converse of Chow's theorem is false in general. Consider for
instance the nonholonomic system in $\R^3$ defined by
$X_1=\partial_x$, $X_2=\partial_y + f(x)\partial_z$ where
$f(x)=e^{-1/x^2}$ for positive $x$ and $f(x)=0$ otherwise. The
associated sub-Riemannian distance is finite whereas $X_1,\dots,X_m$ do
not satisfy Chow's Condition.

\noindent However, for an analytic nonholonomic system (i.e.\ when $M$
and the vector fields $X_1,\dots,X_m$
are in the analytic category), Chow's Condition is
equivalent to the controllability of~\eqref{eq:snh} (see~\cite{nag66,sus74}).
\end{remark}

\begin{remark}
\label{re:orbit}
Our proof of
Theorem~\ref{thmCR} also shows that, under the assumptions
of the theorem, for every point $p \in M$ the set
$$
\left\{ \exp(t_{i_1} {X}_{i_1}) \circ  \cdots \circ \exp(t_{i_k}
{X}_{i_k})(p) \ : \ k\in \N, \ t_{i_j} \in \R, \ i_j \in \{ 1,
\dots, m\} \right\}
$$
is equal to the whole $M$. This set is often called the \emph{orbit}
at $p$ of the vector fields $X_1, \dots, X_m$.
\end{remark}

\subsection{Topological structure of $(M,d)$}
\label{se:topo}

The proof of Lemma~\ref{le:Apopen} gives a little bit more than
the openness of $\mathcal{A}_p$.
For $\eps$ small enough, any $\phi_t^i(q)$, $0\leq t \leq \eps$, is
a trajectory of length $\eps$. Thus $\varphi (t_1, \dots,t_n)$
is the endpoint of a trajectory of length less than $N
\big( |t_1|^{1/|I_1|} + \cdots + |t_n|^{1/|I_n|} \big)$, where $N$
counts the maximal number of concatenations involved in the
$\psi^{I_i}_{t}$'s. This gives an upper bound for the distance,
\begin{eqnarray}
\label{eq:ineqd}
d \big( p, \varphi(t) \big) \leq N \big( |t_1|^{1/|I_1|} + \cdots +
|t_n|^{1/|I_n|} \big).
\end{eqnarray}

This kind of estimates of the distance in terms of local
coordinates plays an important role in sub-Riemannian geometry, as we
will see in Section~\ref{se:distance}. However here
$(t_1,\dots,t_n)$ are not smooth local coordinates, as $\varphi$ is only a
$C^1$-diffeomorphism, not a smooth diffeomorphism.

Let us try to replace $(t_1,\dots,t_n)$ by smooth local coordinates.
Choose local coordinates $(y_1,\dots,y_n)$ centered at $p$ such
that $\frac{\partial}{\partial y_i}|_p=X_{I_i}(p)$. The map
$\varphi^y=y \circ \varphi$ is a $C^1$-diffeomorphism between
neighbourhoods of $0$ in $\R^n$, and its differential at 0 is
$d\varphi^y_0= \mathrm{Id}_{\R^n}$.

Denoting by $\| \cdot \|_{\R^n}$ the Euclidean norm on $\R^n$, we
obtain, for $\|t \|_{\R^n}$ small enough, $y_i(t)=t_i + o(\|t
\|_{\R^n})$. The inequality (\ref{eq:ineqd}) becomes
$$
d(p,q^y) \leq N' \|y \|_{\R^n}^{1/r},
$$
where $q^y$ denotes the point of coordinates $y$, and $r=\max_i
|I_i|$. This inequality allows to compare $d$ to a Riemannian
distance.

Let $g_R$ be a Riemannian metric on $M$, and $d_R$ the associated
Riemannian distance.  On a compact neighbourhood of $p$, there exists
a constant $c>0$ such that  $g(X_i,X_i)(q) \leq c^{-1}$, which implies
$c d_R(p,q) \leq d(p,q)$. Moreover
we have
$d_R(p,q^y) \geq \mathit{Cst} \|y \|_{\R^n}$.  We have then obtained
a first estimate to the sub-Riemannian distance.

\begin{theorem}
\label{th:RvsSR}
Assume \eqref{eq:snh} satisfies Chow's Condition. For any
Riemannian metric $g_R$, we have, for $q$
close enough to $p$,
$$
c d_R(p,q) \leq d(p,q) \leq C d_R(p,q)^{1/r},
$$
where $c,C$ are positive constants and $r$ is an integer such that
$\Delta_p^r=T_pM$.
\end{theorem}

\begin{remark}
 If we choose for $g_R$ a Riemannian metric which is compatible with
 $g$, that is, which satisfies $g_R|_\Delta=g$, then by construction
 $d_R(p,q) \leq d(p,q)$.
\end{remark}

\begin{corollary}
\label{le:d=dist}
Under the hypotheses of Theorem~\ref{thmCR},  $d$ is a distance
function on $M$, i.e.,
\begin{enumerate}[(i)]
\item $d$ is a function from $M \times M$ to $[0,\infty)$;
\item $d(p,q) = d(q,p)$ (symmetry);
\item $d(p,q)=0$ if and only if $p=q$;
\item $d(p,q) + d(p,q') \leq d(p,q')$ (triangle inequality).
\end{enumerate}
\end{corollary}

\begin{proof}
By Chow-Rashevsky's theorem (Theorem~\ref{thmCR}), the distance
between any pair of points is finite, which gives \emph{(i)}.
The symmetry of the distance results from the fact that, if
$\gamma(s)$, $s\in [0,T]$, is a trajectory joining $p$ to $q$, then
$s \mapsto \gamma(T-s)$ is a trajectory of same length  joining $q$
to $p$. Point \emph{(iii)} follows directly from
Theorem~\ref{th:RvsSR}. Finally,
the triangle inequality is a consequence of the following remark. If
$\gamma(s)$, $s\in [0,T]$,  is a trajectory joining $p$ to $q$ and
$\gamma'(s)$, $s\in [0,T']$, is a trajectory
joining $q$ to $q'$, then the concatenation $\gamma * \gamma'$,
defined by
$$
\gamma * \gamma' (s) =
\left\{
  \begin{array}[c]{ll}
    \gamma(s) & \hbox{if } s\in [0,T], \\
\gamma'(s-T) & \hbox{if } s\in [T,T+T'], \\
  \end{array}
\right.
$$
 is a
trajectory joining $p$ to $q'$ whose length satisfies
$$
{\mathrm{length}} (\gamma * \gamma') = {\mathrm{length}} (\gamma) +
{\mathrm{length}} (\gamma').
$$
\end{proof}

A second consequence of Theorem~\ref{th:RvsSR} is that the
sub-Riemannian distance $d$ is $1/r$-H\"{o}lder
with respect to any Riemannian distance, and so continuous.

\begin{corollary}
\label{le:toposr}
If \eqref{eq:snh} satisfies Chow's Condition, then the topology
of the metric space $(M,d)$ coincides with the topology of $M$ as a
smooth manifold.
\end{corollary}



\section{First-order approximations}
\label{se:1storder}

Consider  a nonholonomic system \eqref{eq:snh}: $\dot{q}=\sum_{i=1}^{m}u_iX_i(q)$ on a manifold $M$ satisfying
Chow's Condition, and denote by $d$ the induced sub-Riemannian distance.
As we have seen in Section~\ref{se:snh}, the infinitesimal behaviour of this system should be captured by an
approximation to the first-order with respect to $d$.
In this section we will then provide notion of first-order
approximation and construct the basis of an infinitesimal calculus
adapted to nonholonomic systems. To this aim, a fundamental role will be played by the concept of noholonomic order of a function at a point. We will then see that
approximations to the first-order appear as nilpotent approximations,
in the sense that $X_1, \dots, X_m$ are approximated by vector fields that
generate a nilpotent Lie algebra.\medskip

The whole section is concerned with local objects. Henceforth,
throughout the section we fix a point $p\in M$ and an open
neighbourhood $U$ of $p$ that we identify
with a neighbourhood of $0$ in $\R^n$ through some local coordinates.


\subsection{Nonholonomic orders}
\label{se:nhorders}
\begin{definition}
Let $f: M \to \R$ be a continuous function. The \emph{nonholonomic order of $f$ at $p$}, denoted by
$\mathrm{ord}_p(f)$, is the real number defined by
$$
\mathrm{ord}_p(f) = \sup \left\{ s \in \R \ : \ f(q) = O \big( d(p,q)^s \big) \right\}.
$$
\end{definition}
This order is always nonnegative. Moreover $\mathrm{ord}_p(f) = 0$ if $f(p) \neq 0$, and $\mathrm{ord}_p(f)
= +\infty$ if $f(p) \equiv 0$.

\begin{example}[Euclidean case]
\label{ex:euclidean}
When $M =\R^n$, $m=n$, and $X_i=\partial_{x_i}$, the sub-Riemannian
distance is simply the  Euclidean distance on $\R^n$. In this case,
nonholonomic orders coincide with the standard ones. Namely,
$\mathrm{ord}_0(f)$  is the smallest
degree of  monomials having nonzero coefficient
in the Taylor series
$$
f(x) \sim \sum c_\alpha x_1^{\alpha_1} \dots x_n^{\alpha_n}
$$
of $f$ at $0$. We will see below that there exists in general an
analogous characterization of nonholonomic orders.
\end{example}

Let $C^\infty (p)$ denote the set of germs of smooth functions at
$p$. For $f \in C^\infty (p)$, we call {\em nonholonomic derivatives
  of order 1 of $f$}  the Lie
derivatives $X_1f, \dots , X_mf$. We call further $X_i(X_jf)$,
$X_i(X_j(X_kf))$,\dots \  the {\em nonholonomic derivatives of $f$ of order
2, 3,\dots}\  The nonholonomic derivative of order 0 of $f$ at $p$ is $f(p)$.

\begin{proposition}
\label{le:order}
Let $f \in C^\infty (p)$. Then $\mathrm{ord}_p(f)$ is equal to the biggest integer $k$ such that all
nonholonomic derivatives of $f$ of order smaller than $k$
  vanish at $p$. Moreover,
  $$
  f(q) = O \big( d(p,q)^{\mathrm{ord}_p(f)} \big).
  $$
\end{proposition}

\begin{proof}
The proposition results from the  following two assertions:
\begin{enumerate}
  \item[\emph{(i)}] if $\ell$ is an integer such that $\ell <\mathrm{ord}_p(f)$, then all nonholonomic
  derivatives of $f$ of order $\leq \ell$ vanish at $p$;
  \item[\emph{(ii)}] if $\ell$ is an integer such that all nonholonomic derivatives of $f$ of order
  $\leq \ell$ vanish at $p$, then $f(q) = O \big( d(p,q)^{\ell + 1} \big)$.
\end{enumerate}
Let us first prove point \emph{(i)}. Let $\ell$ be an integer such that $\ell <\mathrm{ord}_p(f)$.
We write a
  nonholonomic
  derivative of
  $f$ of order $k \leq \ell$  as
  $$
(X_{i_1} \dots X_{i_k}f)(p) = \frac{\partial^k}{\partial t_1
\cdots
\partial t_k} f \big( \exp(t_k X_{i_k}) \circ \cdots \circ \exp(t_1 X_{i_1}) (p)
\big)\Big|_{t=0}.
  $$
The point $q =\exp(t_k X_{i_k}) \circ \cdots \circ \exp(t_1 X_{i_1})
(p)$ is the endpoint of a trajectory of length $|t_1|+
\cdots + |t_n|$. Therefore, $d(p,q) \leq |t_1|+ \cdots + |t_n|$.

Since $k \leq \ell <\mathrm{ord}_p(f)$, there exists a real number $s>0$ such that $f(q)= O\big((|t_1|+ \cdots +
|t_n|)^{k+s} \big)$. This implies that
$$
(X_{i_1} \dots X_{i_k}f)(p) = \frac{\partial^k}{\partial t_1
\cdots
\partial t_k} f (q)\Big|_{t=0} =0.
$$
Thus point \emph{(i)} is proved.\medskip

The proof of point \emph{(ii)} goes by induction on $\ell$.
For $\ell=0$, assume that all nonholonomic derivatives of $f$ of order $\leq 0$ vanish at $p$, that is $f(p)=0$.
Choose any Riemannian metric on $M$ and denote by $d_R$ the associated Riemannian distance on $M$.
Since $f$ is smooth, there holds $f(q) \leq \mathit{Cst} \ d_R(p,q)$ near $p$. By Theorem~\ref{th:RvsSR},  this implies
$f(q) \leq \mathit{Cst} \ d(p,q)$, and so property \emph{(ii)} for $\ell = 0$.

Assume that, for a given $\ell \geq 0$, \emph{(ii)} holds for any function $f$ (induction
hypothesis) and take a function $f$ such that all its nonholonomic
derivatives of order $<\ell+1$ vanish at $p$.

Observe that, for $i=1, \dots,m$, all the nonholonomic derivatives of
$X_if$ of order $<\ell$ vanish at $p$. Indeed, $X_{i_1} \dots
X_{i_k}(X_if) = X_{i_1} \dots X_{i_k}X_if$. Applying the induction
hypothesis to $X_if$ leads to $X_if(q)= O \big( d(p,q)^\ell \big)$. In
other words, there exist positive constants $C_1, \dots, C_m$ such
that, for $q$ close enough to $p$,
$$
X_if (q) \leq C_i d(p,q)^\ell.
$$

Fix now a point $q$ near $p$. By Remark~\ref{re:velocity1}, there
exists a minimizing
curve $\gamma (\cdot)$ of
velocity one joining $p$ to $q$.
Therefore $\gamma$ satisfies
$$
\dot \gamma (t) = \sum_{i=1}^m u_i(t) X_i \big(\gamma(t)\big)
\quad \hbox{for a.e.\ }t \in [0,T], \qquad \gamma(0)=p, \
\gamma(T)=q,
$$
with $\sum_i u_i^2 (t) =1$ a.e.\ and $d\big(p,\gamma (t)\big) = t$
for any $t \in [0,T]$. In particular $d(p,q)= T$.

To estimate $f(q)=f\big(\gamma(T)\big)$, we compute the derivative of
$f\big(\gamma(t)\big)$ with respect to $t$,
\begin{eqnarray*}
\frac{d}{dt}f\big(\gamma(t)\big) & = & \sum_{i=1}^m u_i(t) X_if
\big(\gamma(t)\big), \\
 \Rightarrow \ \left| \frac{d}{dt}f\big(\gamma(t)\big) \right| & \leq
 & \sum_{i=1}^m
 |u_i(t)| C_i d\big(p,\gamma (t)\big)^\ell \leq Ct^\ell,
\end{eqnarray*}
where $C=C_1 + \cdots +C_m$. Integrating this inequality between 0
and $t$ gives
$$
\big| f\big(\gamma(t)\big) \big| \leq | f(p)| + \frac{C}{\ell+1}
t^{\ell+1}.
$$
Note that $f(p)=0$, since the nonholonomic derivative of $f$ of order
$0$ at $p$ vanishes. Finally, at $t=T=d(p,q)$, we obtain
$$
| f(q)| \leq \frac{C}{\ell+1} T^{\ell+1},
$$
which concludes the proof of \emph{(ii)}.
\end{proof}

As a consequence, the nonholonomic order of a smooth (germ of)
function is given by the formula
$$
\mathrm{ord}_p(f)= \min \big\{ s \in \N \ : \ \exists \ i_1,\dots,i_s \in
\{1,\dots,m\} \  \ \mathrm{s.t.} \   \ (X_{i_1}  \dots X_{i_s}f)(p) \neq 0
\big\},
$$
where as usual we adopt the convention that $\min \emptyset = + \infty$.

%

It is clear now that any function in $C^\infty (p)$  vanishing at $p$
is of order $\geq 1$. Moreover,  the following basic
computation rules are satisfied: for every $f,g$ in
$C^\infty (p)$ and every  $\lambda \in \R \setminus \{0\}$,
\begin{eqnarray*}
\mathrm{ord}_p(fg) & \geq & \mathrm{ord}_p (f) + \mathrm{ord}_p (g) ,\\
\mathrm{ord}_p(\lambda f) & = & \mathrm{ord}_p (f),\\
\mathrm{ord}_p(f+g) & \geq & \min \big( \mathrm{ord}_p(f),
\mathrm{ord}_p (g) \big).
\end{eqnarray*}
Note that the first inequality is actually an equality. However
the proof of this fact requires an additional result (see
Proposition~\ref{le:algorder}).\medskip

The notion of nonholonomic order extends to vector fields. Let $VF
(p)$ denote the set of germs
of smooth vector fields at $p$.

\begin{definition}
Let $X \in VF (p)$. The \emph{nonholonomic order of $X$ at $p$},
denoted by $\mathrm{ord}_p(X)$,
is the real number defined by:
$$
\mathrm{ord}_p(X) = \sup \left\{ \sigma \in \R \ : \ \mathrm{ord}_p(Xf) \geq \sigma + \mathrm{ord}_p(f),
\quad \forall f \in C^\infty (p) \right\}.
$$
The order of a differential operator is defined in the same way.
\end{definition}

Note that $\mathrm{ord}_p(X) \in \Z$ since the
order of a smooth function is an integer. Moreover the null
vector field $X\equiv0$ has infinite order,
$\mathrm{ord}_p(0)=+\infty$.

Since the order of a function coincides with its order as a
  differential operator acting by multiplication, we have  the
  following properties. For every $X, Y \in VF (p)$ and every
   $f \in C^\infty(p)$,
\begin{equation}
\label{eq:ordcl}
\begin{array}{rcl}
\mathrm{ord}_p([X,Y]) & \geq & \mathrm{ord}_p (X) + \mathrm{ord}_p (Y) ,\\
\mathrm{ord}_p(fX) & \geq & \mathrm{ord}_p (f) + \mathrm{ord}_p (X),\\
\mathrm{ord}_p(X) & \leq & \mathrm{ord}_p (Xf) - \mathrm{ord}_p (f),\\
\mathrm{ord}_p(X+Y) & \geq & \min \big( \mathrm{ord}_p(X), \mathrm{ord}_p (Y) \big).
\end{array}
\end{equation}
As already noticed for functions, the second inequality is in fact
an equality. This is not the case for the first inequality (take
for instance $X=Y$).

As a consequence of~\eqref{eq:ordcl}, $X_1, \dots ,X_m$ are of order $\geq -1$,
$[X_i,X_j]$ of order $\geq -2$, and more generally, every $X$ in the set $\Delta^k$ is of order $\geq -k$.

\begin{example}[Euclidean case]
\label{ex:euclidean2}
In the Euclidean case (see example~\ref{ex:euclidean}),
 the nonholonomic order of a constant differential operator is
the negative of its usual order. For instance $\partial_{x_i}$ is
of nonholonomic order $-1$. Actually, in this case, every vector field
that does not vanish at $p$ is of nonholonomic order $-1$.
\end{example}

\begin{example}[Heisenberg case]
\label{ex:heis}
Consider the following vector fields  on $\R^{3}$:
$$
X_1  = {\partial_x} - \frac{y}{2}  {\partial_z} \quad \mbox{ and } \quad
X_2 = {\partial_y} + \frac{x}{2}  {\partial_z}.
$$
The coordinate functions $x$ and $y$ have order 1 at $0$, whereas $z$
has order 2 at  $0$,  since  $X_1x(0)=X_2y(0)=1$, $X_1z(0)=X_2z(0)=0$, and
$X_1X_2z(0)=1/2$. These relations also imply $\mathrm{ord}_0(X_1) =
\mathrm{ord}_0(X_2) = -1$. Finally, the Lie bracket $[X_1,X_2]
= \partial_{z}$ is of order $-2$ at $0$ since $[X_1,X_2]z=1$.
\end{example}

We are now in a position to give a meaning to first-order approximation.

\begin{definition}
A family of $m$ vector fields $\widehat{X}_1, \dots, \widehat{X}_m$
defined near $p$ is called a {\em
first-order approximation of $X_1, \dots, X_m$ at $p$} if the vector fields
$X_i - \widehat{X}_i$, $i=1, \dots, m$, are of order $\geq 0$ at $p$.
\end{definition}
A consequence of this definition is that the order at $p$ defined by the vector fields
$\widehat{X}_1, \dots, \widehat{X}_m$
coincides with the one defined by $X_1, \dots, X_m$. Hence for any $f \in C^\infty (p)$ of order greater than $s-1$,
$$
(X_{i_1}  \dots X_{i_s}f)(q) = (\widehat{X}_{i_1}  \dots
\widehat{X}_{i_s}f)(q) +  O \left( d(p,q)^{\mathrm{ord}_p(f)-s+1} \right).
$$


To go further in the characterization of orders and
approximations, we need suitable systems of coordinates.

\subsection{Privileged coordinates}
\label{se:privcoor}
We have introduced in Section~\ref{se:chow} the sets of vector
fields $\Delta^s$, defined  by $\Delta^s = \mathrm{span} \{ X_I \ : \ |I|
\leq s \}$. Since $X_1,\dots,X_m$ satisfy Chow's Condition, the
values of these sets at $p$ form a flag of subspaces of $T_pM$, that is,
\begin{eqnarray}
\label{eq:flag}
\Delta^1(p) \subset \Delta^2 (p) \subset \cdots \subset  \Delta^{r-1}(p) \varsubsetneq  \Delta^r(p)=T_pM,
\end{eqnarray}
where $r=r(p)$ is called the {\em degree of nonholonomy at $p$}.

Set $n_i(p)= \dim \Delta^i(p)$. The $r$-tuple of integers $(n_1(p), \dots,
n_r(p))$ is called the {\em growth vector at $p$}.  The first
integer $n_1(p)\leq m $ is the rank of the family
$X_1(p),\dots,X_m(p)$, and
the last
one $n_r(p)=n$ is the dimension of the manifold $M$.

Let $s \geq 1$. By abuse of notations, we continue to write
$\Delta^s$ for the map $q\mapsto \Delta^s (q)$. This map $\Delta^s$ is a distribution if and
only if $n_s(q)$ is constant on $M$. We
then distinguish two kind
of points.

\begin{definition}
The point $p$ is a {\em regular point} if the growth vector is
constant in a neighbourhood of $p$. Otherwise, $p$ is a {\em
singular point}.
\end{definition}
Thus, near a regular point, all maps $\Delta^s$ are locally
distributions.\medskip

The structure of flag~(\ref{eq:flag}) may also be described by
another sequence of integers. We define the  {\em
weights at $p$}, $w_i=w_i(p)$, $i=1, \dots, n$, by setting $w_j=s$
if $n_{s-1}(p) < j \leq n_s(p)$, where $n_0=0$. In other words,
we have
\begin{multline*}
w_1=\cdots=w_{n_1}=1, \ w_{n_1+1}= \cdots=w_{n_2}=2,  \dots ,\\
w_{n_{r-1}+1}=\cdots=w_{n_r}=r.
\end{multline*}
 The weights at $p$ form an increasing sequence $w_1(p) \leq \cdots
 \leq w_n(p)$ which is constant near $p$ if and only if $p$ is a
 regular point.

\begin{example}[Heisenberg case]
\label{ex:heis2}
The Heisenberg case in $\R^3$ given in example~\ref{ex:heis} has
a growth vector which is equal to $(2,3)$ at every point. Therefore all points
of $\R^3$ are regular. The
weights at any point are $w_1=w_2=1$, $w_3=2$.
\end{example}

\begin{example}[Martinet case]
\label{ex:martinet}
Consider the following vector fields on $\R^{3}$,
$$
X_1  = {\partial_x}  \quad \mbox{ and } \quad
X_2 = {\partial_y} + \frac{x^2}{2}  {\partial_z}.
$$
The only nonzero brackets are
$$
X_{12}= [X_1,X_2]  = x{\partial_z}  \quad \mbox{ and } \quad
X_{112}= [X_1, [X_1,X_2]]  =  {\partial_z}.
$$
Thus the growth vector is equal to
 $(2,2,3)$ on the
 plane $\{ x=0\}$, and to $(2,3)$ elsewhere. As a consequence, the set
 of singular points is the plane $\{ x=0\}$. The
weights  are $w_1=w_2=1$, $w_3=2$ at regular points, and $w_1=w_2=1$,
$w_3=3$ at singular ones.
\end{example}

\begin{example}
\label{ex:dnhinfini}
Consider the  vector fields on $\R^3$
$$
X_1=\partial_x \quad \hbox{and} \quad X_2=\partial_y + f(x) \partial_z,
$$
where $f$ is a smooth function on $\R$ which admits every positive
integer $n \in \N$ as a zero with multiplicity  $n$ (such a function
exists and can even be chosen in the analytic class thanks to the
Weierstrass factorization theorem~\cite[Th.\ 15.9]{Rudin1970}). Every
point $(n,y,z)$ is
singular and the weights at this point are $w_1=w_2=1$,
$w_3=n+1$. As a consequence the degree of nonholonomy $w_3$ is
unbounded on $\R^3$.
\end{example}

Let us give some basic properties of the growth vector and of the weights.
\begin{itemize}

  \item At a regular point, the growth vector is a strictly
  increasing sequence: $n_1(p) < \cdots < n_r(p)$. Indeed, if
  $n_s(q)=n_{s+1}(q)$ in a neighbourhood of $p$, then $\Delta^s$ is
  locally an involutive distribution and so $s=r$.
  As a consequence, at a regular point $p$, the jump between two
  successive weights is never greater than 1, $w_{i+1} - w_i \leq 1$,
  and there holds $r(p) \leq n-m+1$.

  \item For every $s$, the map $q \mapsto n_s(q)$ is a lower
    semi-continuous
    function from $M$ to $\N$.  Therefore, the set of regular points
    is open and dense in $M$.

  \item For every $i=1, \dots, n$, the weight $w_i(\cdot)$ is an upper
    semi-continuous function. In particular, this is the case for the
    degree of nonholonomy $r(\cdot)=w_n(\cdot)$, that is,
  $r(q) \leq r(p)$ for $q$ near $p$. As a consequence $r(\cdot)$ is bounded on any compact subset of $M$.

\item The degree of nonholonomy may be
  unbounded on $M$ (see example~\ref{ex:dnhinfini} above). Thus determining if
   a nonholonomic system is controllable is a non
  decidable problem: the computation of an infinite number of brackets may be
  needed to decide if Chow's Condition is satisfied.

  However, in the case of polynomial vector fields  on
  $\R^n$  (relevant in practice), it can be shown that
  the degree of nonholonomy
  is bounded by a universal function of the degree $k$ of the
  polynomials (see~\cite{gab95b,gab98}):
  $$
r(x) \leq 2^{3n^2} n^{2n} k^{2n}.
  $$
\end{itemize}

The meaning of the sequence of weights is best understood in terms of basis
of $T_pM$. Choose first vector fields $Y_1, \dots,Y_{n_1}$ in
$\Delta^1$ whose values at $p$ form a
basis of $\Delta^1(p)$.  Choose then vector fields $Y_{n_1+1}, \dots,Y_{n_2}$
in $\Delta^2$ such that the values
$Y_1(p),\dots,Y_{n_2}(p)$ form a
basis of $\Delta^2(p)$. For  each $s$, choose $Y_{n_{s-1}+1},
\dots,Y_{n_s}$ in $\Delta^s$
such that $Y_1(p),\dots,Y_{n_s}(p)$ form a basis of $\Delta^s(p)$. We
obtain in this way a
family of vector fields $Y_1, \dots , Y_n$ such that
\label{pa:adaptedframe}
\begin{eqnarray}
\label{eq:adaptedframe}
\left\{
\begin{array}{l}
  Y_1(p), \dots , Y_{n}(p) \hbox{ is a basis of } T_pM, \\
  Y_i \in \Delta^{w_i} , \ i=1, \dots, n.
\end{array}
\right.
\end{eqnarray}
A family of $n$ vector fields satisfying (\ref{eq:adaptedframe}) is
called an {\em adapted frame  at $p$}. The word ``adapted''
means here ``adapted to the flag~(\ref{eq:flag})'', since the values at
$p$ of an adapted frame contain a basis $Y_1(p), \dots,
Y_{n_s}(p)$ of each subspace $\Delta^s(p)$ of the flag. By continuity,
at a point $q$ close enough to $p$,
the values of
$Y_1, \dots, Y_n$  still form  a basis of $T_qM$.
However, if $p$ is singular,  this basis may  not be adapted to the
flag~(\ref{eq:flag}) at $q$.\medskip

Let us explain now the relation between weights and orders. We write first the
tangent space as a direct sum,
$$
T_pM= \Delta^1(p) \oplus \Delta^2(p) / \Delta^1(p) \oplus \cdots \oplus
\Delta^r(p)/ \Delta^{r-1}(p),
$$
where $\Delta^s(p)/\Delta^{s-1}(p)$ denotes a supplementary of
$\Delta^{s-1}(p)$ in
$\Delta^s(p)$, and take a local system of coordinates $(y_1,\dots,y_n)$.
The dimension of each space $\Delta^s(p)/\Delta^{s-1}(p)$ is equal to
$n_s-n_{s-1}$,
and we can assume that, up to a reordering, we have
$dy_j(\Delta^s(p)/\Delta^{s-1}(p)) \neq 0$ for $n_{s-1}< j \leq n_s$.

Take an integer $j$ such that  $0<j \leq n_1$. From the assumption
above, there holds $dy_j(\Delta^1(p))\neq 0$, and consequently there
exists  $X_i$ such that $dy_j(X_i(p)) \neq 0$. Since
$dy_j(X_i)=X_i y_j$ is a first-order nonholonomic derivative of $y_j$, we
have $\mathrm{ord}_p (y_j) \leq 1=w_j$.

Take now an integer $j$ such that $n_{s-1}< j \leq n_s$ for $s >1$,
that is, $w_j=s$.
Since $dy_j(\Delta^s(p)/\Delta^{s-1}(p)) \neq 0$, there exists a
vector field $Y$ in $\Delta^s$ such that  $dy_j(Y(p))=(Yy_j)(p) \neq 0$. By definition of
$\Delta^s$, the Lie derivative $Yy_j$ is a
linear combination of nonholonomic derivatives of $y_j$ of order not
greater than $s$. One of them must be nonzero, and so
$\mathrm{ord}_p(y_j) \leq s=w_j$.

Finally, any system of local coordinates $(y_1,\dots,y_n)$
satisfies $\mathrm{ord}_p(y_j) \leq w_j$ up to a reordering (or
$\sum_{i=1}^n \mathrm{ord}_p(y_i) \leq \sum_{i=1}^n
w_i$ without
reordering).  The coordinates with the maximal possible order will play an
important role.

\begin{definition}
A {\em system of privileged coordinates at $p$} is a system of
local coordinates $(z_1,\dots,z_n)$ such that $\mathrm{ord}_p(z_j)=w_j$
for $j=1,\dots,n$.
\end{definition}

Notice that privileged coordinates $(z_1,\dots,z_n)$ satisfy
\begin{eqnarray}
\label{eq:adapted}
dz_i (\Delta^{w_i}(p)) \neq 0, \quad dz_i (\Delta^{w_i-1}(p))=0, \quad
i=1,\dots,n,
\end{eqnarray}
or, equivalently, $\partial_{z_i}|_p$ belongs to $\Delta^{w_i}(p)$ but
not to $\Delta^{w_i-1}(p)$. Local coordinates satisfying~(\ref{eq:adapted}) are
called {\em linearly adapted coordinates} (``adapted" because the
differentials at $p$ of the coordinates form a basis of $T^*_pM$
dual to the values of an adapted frame).
Thus privileged coordinates are always linearly adapted
coordinates. The converse is false, as shown in the example below.

\begin{example}
Take $X_1=\partial_x$, $X_2=\partial_y + (x^2 +y) \partial_z$ in
$\R^3$. The weights at 0 are $(1,1,3)$ and $(x,y,z)$ are adapted
at 0. But they are not privileged: indeed, the coordinate $z$ is of order
$2$ at $0$ since $(X_2X_2 z)(0)=1$.
\end{example}

\begin{remark}
As it is suggested by Kupka~\cite{kup96}, one can define
\emph{privileged functions at $p$} to be the smooth
functions $f$ on $U$ such that
$$
\mathrm{ord}_p(f) = \min \{ s \in \N \ : \ df(\Delta^s(p)) \neq 0 \}.
$$
It results from the discussion above that some
local coordinates $(z_1,\dots,z_n)$ are privileged
 at $p$ if and only if each $z_i$ is a privileged function
at $p$.
\end{remark}

Let us now show how to compute orders using privileged
coordinates.
We fix a system of privileged coordinates $(z_1,\dots,z_n)$ at $p$.
Given a sequence of integers $\alpha=(\alpha_1, \dots, \alpha_n)$,
we define the weighted degree of the monomial
$z^\alpha=z_1^{\alpha_1} \cdots z_n^{\alpha_n}$ to be  $w(\alpha)=w_1
\alpha_1 + \cdots + w_n \alpha_n$ and the weighted degree of the
monomial vector field $z^\alpha \partial_{z_j}$ as
$w(\alpha)-w_j$. The weighted degrees allow to compute the orders
of functions and vector fields in a purely algebraic way.

\begin{proposition}
\label{le:algorder}
For a smooth function $f$ with a Taylor expansion
$$
f(z) \sim \sum_{\alpha} c_\alpha z^\alpha,
$$
the order of $f$ is the least weighted degree of  monomials having  a
nonzero coefficient in the Taylor series.

For a vector field $X$ with a Taylor expansion
  $$
X(z) \sim \sum_{\alpha, j} a_{\alpha,j} z^\alpha \partial_{z_j},
$$
the order of $X$ is the least weighted degree of a monomial vector
fields  having  a
nonzero coefficient in the Taylor series.
\end{proposition}
In other words, when using privileged coordinates, the notion of nonholonomic
order amounts to the usual notion of vanishing order at some
point, only assigning weights to the variables.

\begin{proof}
For $i=1,\dots,n$, we have $\partial_{z_i}|_p \in \Delta^{w_i}
(p)$. Then there exist $n$ vector fields $Y_1, \dots,Y_n$ which form
an adapted
frame at $p$ and   such that
$Y_1(p)=\partial_{z_1}|_p$, \dots,  $Y_n(p)=\partial_{z_n}|_p$. For
every $i$, the vector field $Y_i$ is of order $\geq -w_i$ at $p$ since
it belongs to
$\Delta^{w_i}$. Moreover we have $(Y_iz_i)(p)=1$ and
$\mathrm{ord}_p(z_i)=w_i$. Thus
$\mathrm{ord}_p(Y_i)=-w_i$.

Take a sequence of integers $\alpha=(\alpha_1, \dots, \alpha_n)$.
The  monomial $z^\alpha$ is of order $\geq w(\alpha)$ at $p$ and
the differential operator $Y^\alpha=Y^{\alpha_1}_1 \cdots
Y^{\alpha_n}_n$ is of order $\geq -w(\alpha)$. Observing that
$(Y_iz_j)(p)=0$ if $j\neq i$, we easily see that $(Y^\alpha
z^\alpha)(p) =\frac{1}{\alpha_1! \dots \alpha_n!}\neq 0$, whence
$\mathrm{ord}_p(z^\alpha)= w (\alpha)$.

In the same way, we obtain that, if $z^\alpha$, $z^\beta$ are two
different monomials and $\lambda$, $\mu$ two nonzero real numbers,
then $\mathrm{ord}_p( \lambda z^\alpha + \mu z^\beta) = \min
\big( w(\alpha), w(\beta) \big)$.
Thus the order of a series is the least weighted degree of
monomials actually appearing in the series itself. This shows the
statement on order
of functions.

As a consequence, for any smooth function $f$, the order at $p$ of
$\partial_{z_i}f$ is $\geq \mathrm{ord}_p(f) -w_i$. Since moreover
$\partial_{z_i}z_i=1$, we obtain that $\mathrm{ord}_p(\partial_{z_i})$
is equal to $-w_i$. The second part of the statement
follows. 
\end{proof}

A notion
of homogeneity is also naturally associated with a system of
privileged coordinates
$(z_1,\dots,z_n)$ defined on $U$.  We define
first the one-parameter family of dilations
$$
\delta_t : (z_1,\dots,z_n) \mapsto (t^{w_1}z_1,\dots, t^{w_n}z_n),
\qquad t \geq 0.
$$
Each dilation $\delta_t$ is a map from $\R^n$ to $\R^n$. By abuse of
notations, for $q\in U$ and  $t$ small
enough we write $\delta_t(q)$ instead of $\delta_t(z(q))$, where
$z(q)$ are  the coordinates of $q$. A dilation $\delta_t$ acts also on
functions and vector fields by
pull-back: $\delta_t^*f=f \circ \delta_t$ and $\delta_t^*X$ is the
vector field such that $(\delta_t^*X)(\delta_t^*f)=\delta_t^*(Xf)$.

\begin{definition}
A function $f$ is {\em homogeneous of degree} $s$ if
$\delta_t^*f=t^sf$. A vector field $X$ is {\em homogeneous of
degree} $\sigma$ if $\delta_t^*X=t^\sigma X$.
\end{definition}
For a smooth function (resp. a smooth vector field), this is the same
as being a finite sum of monomials (resp. monomial vector fields)
of weighted degree $s$. As a consequence, if a function $f$ is homogeneous of degree $s$,
then it is of order $s$ at $p$.

A typical degree $1$ homogeneous function is the so-called {\em
  pseudo-norm at $p$}, defined by:
\begin{equation}
  \label{eq:pseudonorm}
z \mapsto \|z\|_p = |z_1|^{1/w_1} + \cdots
+|z_n|^{1/w_n}.
\end{equation}
When composed with the coordinates function, the pseudo-norm at $p$ is a (non
smooth) function of order $1$, that is,
$$
\|z(q)\|_p = O \big( d(p,q)  \big).
$$
Actually, it results from Proposition~\ref{le:algorder} that the order
of a function $f\in C^\infty (p)$
 is the least integer $s$ such that $f(q)=O (\|z(q)\|_p^s)$.

\paragraph{Examples of privileged coordinates.}

Of course all the results above on algebraic computation of orders
hold only if privileged coordinates do exist. Two types of
privileged coordinates are commonly used in the literature.

\subparagraph{a. Exponential coordinates.} Choose an adapted frame
$Y_1, \dots, Y_n$ at $p$. The inverse of the local
diffeomorphism
  $$
(z_1, \dots, z_n) \mapsto \exp(z_1 Y_1 + \cdots + z_n Y_n)(p)
  $$
  defines a system of local privileged coordinates at $p$, called
  {\em canonical coordinates of the first kind}. These coordinates
  are mainly used in the context of hypoelliptic operator and for
  nilpotent Lie groups with right (or left) invariant sub-Riemannian
  structure.

The inverse of the local diffeomorphism
  $$
(z_1, \dots, z_n) \mapsto \exp(z_n Y_n)\circ \cdots \circ \exp(z_1
Y_1)(p)
  $$
 also defines privileged coordinates at $p$, called
  {\em canonical coordinates of the second kind}.
They are easier to work with than the one of the first kind. For
instance, in these coordinates, the vector field $Y_n$ read as
$\partial_{z_n}$. One can also exchange the order of the flows in
the definition to obtain any of the $Y_i$ as $\partial_{z_i}$. The
fact that canonical coordinates of both first and second kind are privileged is proved in Section~\ref{se:2ndkind}.

We leave it to the reader to verify that the diffeomorphism
 $$
(z_1, \dots, z_n) \mapsto \exp(z_n Y_n+ \cdots +z_{s+1}Y_{s+1})\circ
\exp(z_s Y_s) \cdots \circ \exp(z_1 Y_1)(p)
  $$
 also induces privileged coordinates. As a matter of fact, any ``mix" between
  first and second kind canonical coordinates defines privileged
  coordinates.

\subparagraph{b. Algebraic coordinates.} There exist also effective
constructions of privileged coordinates (the construction of
exponential coordinates is not effective in general since it
requires to integrate flows of vector fields). We present here Bella\"{\i}che's
algorithm, but other constructions exist (see~\cite{ste86,agr87}).
\label{pa:coordalg}

\begin{enumerate}
  \item Choose an adapted frame $Y_1, \dots, Y_n$ at $p$.
  \item Choose coordinates $(y_1,\dots,y_n)$ centered at $p$ such
  that $\partial_{y_i}|_p = Y_i(p)$.
  \item For $j=1, \dots,n$, set
  $$
z_j = y_j - \sum_{k=2}^{w_j-1} h_k(y_1, \dots, y_{j-1}),
 $$
  where, for $k=2, \dots, w_j-1$,
  $$
h_k(y_1, \dots, y_{j-1})=
\sum_{\stackrel{\scriptstyle{|\alpha|=k}}{w(\alpha)<w_j}} \!\!\!\!
Y_1^{\alpha_1} \ldots Y_{j-1}^{\alpha_{j-1}} \Big(
y_j -\sum_{q=2}^{k-1} h_q(y)\Big)(p) \
\frac{y_1^{\alpha_1}}{\alpha_1!} \cdots
\frac{y_{j-1}^{\alpha_{j-1}}}{\alpha_{j-1}!},
  $$
  with $|\alpha|=\alpha_1 + \cdots + \alpha_n$.
\end{enumerate}
The fact that coordinates $(z_1, \dots, z_n)$ are privileged at $p$ will
be proved in Section~\ref{se:coordalg}.

Coordinates $(y_1,\dots,y_n)$ are linearly adapted
  coordinates. They can be obtained from any original system of
  coordinates by an affine change. The privileged coordinates $(z_1, \dots, z_n)$ are then obtained from
  $(y_1,\dots,y_n)$ by an expression of the form
  \begin{eqnarray*}
   z_1 & = & y_1, \\
   z_2 & = & y_2 + \mathrm{pol}(y_1), \\
   & \vdots & \\
   z_n & = & y_n + \mathrm{pol}(y_1, \dots, y_{n-1}),
  \end{eqnarray*}
  where each $\mathrm{pol}$ is a polynomial function without constant
  nor linear
  terms. The inverse change of coordinates takes the same
  triangular form, which makes the use of these coordinates easy for computations.

\subsection{Nilpotent approximation}
Fix a system of privileged coordinates $(z_1, \dots, z_n)$ at $p$. Every vector field $X_i$ is of order $\geq -1$, hence it has, in $z$ coordinates,  a  Taylor expansion
$$
X_i(z) \sim \sum_{\alpha, j} a_{\alpha,j} z^\alpha \partial_{z_j},
$$
where $w(\alpha) \geq w_j-1$ if $a_{\alpha,j} \neq 0$.
Grouping together the monomial vector fields of same weighted degree,
we express $X_i$ as a series
$$
X_i = X_i^{(-1)} + X_i^{(0)} + X_i^{(1)} + \cdots
$$
where $X_i^{(s)}$ is a homogeneous vector field of degree $s$.

\begin{proposition}
\label{le:hmnilp}
Set $\widehat{X}_i=X_i^{(-1)}$, $i=1,\dots,m$. The family of vector
fields $\widehat{X}_1, \dots, \widehat{X}_m$ is a
first-order approximation of $X_1, \dots, X_m$ at $p$ and generate a
nilpotent Lie algebra of step $r=w_n$.
\end{proposition}
\begin{proof}
The fact that the vector fields  $\widehat{X}_1, \dots, \widehat{X}_m$ form a
first-order approximation of $X_1, \dots, X_m$ results from their
construction.

Note further that any homogeneous vector field of degree smaller
than $-w_n$ is zero, as it is easy to check in privileged
coordinates. Moreover,
if $X$ and $Y$ are homogeneous of degree respectively $k$ and $l$,
then the bracket $[X,Y]$ is homogeneous of degree $k+l$ because
$\delta_t^* [X,Y]= [\delta_t^*X, \delta_t^*Y]=t^{k+l}[X,Y]$.

It follows that every iterated bracket of the vector fields $\widehat{X}_1,
\dots, \widehat{X}_m$ of length $k$ (i.e.\ containing $k$ of these
vector field) is homogeneous of degree $-k$ and is zero if $k >
w_n$. 
\end{proof}

\begin{definition}
The family $(\widehat{X}_1, \dots, \widehat{X}_m)$ is called the {\em
  (homogeneous)
nilpotent approximation} of $(X_1, \dots, X_m)$ at $p$ associated with the
coordinates $z$.
\end{definition}

\begin{example}[unicycle]
\label{ex:unicycle2}
Consider the vector fields on $\R^2 \times S^1$ defining the kinematic
model of a unicycle (see example~\ref{ex:unicycle}), that is,
$X_1=\cos \theta \partial_x + \sin \theta \partial_y$,
$X_2= \partial_\theta$. We have $[X_1,X_2]=\sin \theta \partial_x - \cos \theta \partial_y$, so the weights are $(1,1,2)$ at every point.  At $p=0$, the coordinates
$(x,\theta)$ have order $1$ and $y$ has order 2, consequently
$(x,\theta,y)$ is a system of privileged coordinates at $0$.  Taking
the Taylor expansion of $X_1$ and $X_2$ in the latter coordinates, we
obtain the homogeneous components:
$$
X_1^{(-1)}= \partial_x + \theta \partial_y, \quad X_1^{(0)}= 0, \quad
X_1^{(1)}= - \frac{\theta^2}{2} \partial_x -
\frac{\theta^3}{3!} \partial_y, \quad \dots
$$
and $X_2^{(-1)}=X_2=\partial_\theta$. Therefore the homogeneous nilpotent
approximation
of $(X_1, X_2)$ at $0$ in coordinates $(x,\theta,y)$ is
$$
\widehat{X}_1= \partial_x + \theta \partial_y, \quad
\widehat{X}_2=\partial_\theta.
$$
We easily check that the  Lie brackets of length $3$ of these vectors
are zero, that is,
$[\widehat{X}_1,[\widehat{X}_1,\widehat{X}_2]] =
[\widehat{X}_2,[\widehat{X}_1,\widehat{X}_2]]=0$, and so the Lie algebra
$\mathit{Lie} (\widehat{X}_1,\widehat{X}_2 )$ is
nilpotent of step $2$.
\end{example}


The homogeneous nilpotent approximation is not intrinsic to the frame
$(X_1,\dots,X_m)$, since it
depends on the chosen system of privileged coordinates. However,
if $\widehat{X}_1, \dots, \widehat{X}_m$ and $\widehat{X}'_1, \dots,
\widehat{X}'_m$ are the nilpotent
approximations associated with two different systems of coordinates,
then their Lie algebras $\mathit{Lie} (\widehat{X}_1, \dots,
\widehat{X}_m )$ and $\mathit{Lie} (\widehat{X}'_1, \dots,
\widehat{X}'_m )$ are isomorphic. If moreover $p$ is a regular point,
then
$\mathit{Lie} (\widehat{X}_1, \dots, \widehat{X}_m )$ is isomorphic
to the graded nilpotent Lie algebra
$$
\mathrm{Gr}(\Delta)_p = \Delta (p) \oplus (\Delta^2 / \Delta^1)(p)
\oplus \cdots \oplus
(\Delta^{r-1} / \Delta^r)(p).
$$

\begin{remark}
The nilpotent approximation denotes in fact two different
objects. Each $\widehat{X}_i$ can be
seen as a vector field on $\R^n$ or as the representation in $z$
coordinates of the vector field $z^*\widehat{X}_i$ defined on a
neighbourhood of $p$ in $M$. This will cause no confusion since
the nilpotent approximation is
associated with a given system of privileged coordinates.
\end{remark}

It is worth to notice the particular form of the nilpotent
approximation  in privileged
coordinates. Write $\widehat{X}_i=\sum_{j=1}^n f_{ij} (z) \partial_{z_j}$. Since
$\widehat{X}_i$ is homogeneous of degree $-1$ and $\partial_{z_j}$ of
degree $-w_j$, the function  $f_{ij}$ is a homogeneous polynomial of
weighted degree $w_j-1$. In particular it can
not involve variables of weight greater than $w_j-1$, that is,
$$
\widehat{X}_i (z)=\sum_{j=1}^n f_{ij} (z_1, \dots, z_{n_{w_j}-1})
\partial_{z_j}.
$$
The nonholonomic control system $\dot z = \sum_{i=1}^m u_i \widehat{X}_i(z)$
associated with the nilpotent approximation is then polynomial and in a
triangular form,
$$
\dot z_j = \sum_{i=1}^m u_i f_{ij} (z_1, \dots, z_{n_{w_j}-1}).
$$
Computing the trajectories of a system in such a form is rather easy:
given the input function
$(u_1(t), \dots,u_m(t))$, it is possible to compute the
coordinates $z_j$ one after the other, only by integration.\medskip

As vector fields on $\R^n$, $\widehat{X}_1, \dots, \widehat{X}_m$
generate a sub-Riemannian distance on $\R^n$ which is homogeneous
with respect to the dilation $\delta_t$.
\begin{lemma} \
\label{le:nahomo}
\begin{itemize}
  \item[(i)] The family $( \widehat{X}_1, \dots, \widehat{X}_m )$
    satisfies Chow's Condition on $\R^n$.

\item[(ii)] The growth vector at $0$  of $(\widehat{X}_1, \dots, \widehat{X}_m
  )$ is equal to
  the one  at $p$ of  $(X_1, \dots, X_m)$.
\end{itemize}
\noindent Let $\widehat{d}$ be
the sub-Riemannian distance on $\R^n$ associated with $( \widehat{X}_1,
\dots, \widehat{X}_m )$.
\begin{itemize}
  \item[(iii)] \  The distance $\widehat{d}$ is homogeneous of degree 1,
  $$
\widehat{d} (\delta_tx, \delta_ty) = t \widehat{d} (x,y).
$$
  \item[(iv)] \ There exists a constant $C >0$ such that, for all $z
  \in \R^n$,
  $$
\frac{1}{C} \|z\|_p \leq \widehat{d}(0,z) \leq C \| z \|_p,
  $$
\end{itemize}
where $\| \cdot \|_p$ denotes the pseudo-norm at $p$
(see~\eqref{eq:pseudonorm}).
\end{lemma}

\begin{proof}
Through the coordinates $z$ we identify the neighbourhood $U$ of $p$
in $M$
with a neighbourhood of $0$ in $\R^n$.

For every iterated
  bracket $X_I=[X_{i_k},\dots,[X_{i_2},X_{i_1}]]$  of the
  vector fields $X_1,\dots, X_m$, we set
$\widehat{X}_I=[\widehat{X}_{i_k},\dots,[\widehat{X}_{i_2},\widehat{X}_{i_1}]]$,
and for $k\geq 1$ we set
  $\widehat{\Delta}^k = \mathrm{span} \{ \widehat{X}_I \ : \ |I| \leq
  k \}$. As noticed in the proof of
Proposition~\ref{le:hmnilp}, a bracket $\widehat{X}_I$ of length
$|I|=k$ is homogeneous of weighted
degree $-k$, and by construction of the nilpotent approximation, there
holds $X_I=\widehat{X}_I +$ terms of order
$>-k$. Therefore,
$$
\widehat{X}_I(0) = X_I(p) \bmod \mathrm{span} \{ \partial_{z_j}\big|_p \: : \: w_j
< k \} = X_I(p) \bmod \Delta^{k-1}(p).
$$
As a consequence, for any integer $k\geq 1$, we have
\begin{equation}
  \label{eq:grvec}
  \dim \widehat{\Delta}^k(0) = \dim \Delta^k(p),
\end{equation}
and property $(ii)$ follows.
Moreover, if $X_{I_1},\dots,X_{I_n}$ form an adapted frame at $p$, then the family $(\widehat{X}_{I_1}(0), \dots, \widehat{X}_{I_n}(0))$ is of rank $n$, which implies that its determinant is nonzero. Since
the determinant of $X_{I_1},\dots,X_{I_n}$ is an homogeneous polynomial of weighted degree $0$, it is nonzero
everywhere,  which implies $(i)$.
%

As for the property $(iii)$, consider the nonholonomic system defined by the nilpotent approximation, that is, $\dot z = \sum_{i=1}^m u_i \widehat{X}_i (z)$.
Observe that, if $\widehat{\gamma}$ is a trajectory of this system,
that is, if
$$
\dot{\widehat{\gamma}} (t)= \sum_{i=1}^m u_i \widehat{X}_i
\big(\widehat{\gamma}(t)\big), \quad t
\in [0,T],
$$
then the dilated curve $\delta_{\lambda} \widehat{\gamma}$ satisfies
$$
\frac{d}{dt}\delta_{\lambda} \widehat{\gamma} (t)= \sum_{i=1}^m \lambda u_i
\widehat{X}_i
\big(\delta_{\lambda} \widehat{\gamma}(t)\big), \quad t
\in [0,T].
$$
Thus $\delta_{\lambda} \widehat{\gamma}$ is  a trajectory of the same system, with extremities
$(\delta_{\lambda} \widehat{\gamma}) (0) = \delta_{\lambda}
(\widehat{\gamma} (0))$ and
$(\delta_{\lambda} \widehat{\gamma}) (T) = \delta_{\lambda}
(\widehat{\gamma} (T))$, and its length equals $\lambda {\mathrm{length}}
(\widehat{\gamma})$.  This
proves the homogeneity of $\widehat{d}$.

Finally, since $( \widehat{X}_1, \dots, \widehat{X}_m )$
    satisfies Chow's Condition, the
distance $\widehat{d}(0,\cdot)$ is continuous on $\R^n$ (see Corollary~\ref{le:toposr}).
 We can then choose a real number
$C >0$ such that, on the compact set $\{ \|z \| =
 1 \}$, we have $1/C \leq \widehat{d} (0,z) \leq C$. Both functions
 $\widehat{d}(0,z)$ and  $\|
 z \|$ being homogeneous of degree 1, the inequality of Property
 $(iv)$ follows. 
\end{proof}


\subsection{Distance estimates}
\label{se:distance}
As it is the case for Riemannian distances, in general it is
impossible to compute analytically a sub-Riemannian distance (it would
require to determine all minimizing curves). This is
 very important to obtain estimates of the distance, at least
locally. In a Riemannian manifold $(M,g)$, the situation is rather
simple: in local coordinates $x$ centered at a point $p$, the Riemannian
distance $d_R$ satisfies:
$$
d_R(q,q')= \| x(q) - x(q') \|_{g_p} + o(\| x(q) \|_{g_p}+ \| x(q') \|_{g_p}),
$$
where $\| \cdot \|_{g_p}$ is the Euclidean norm induced by the value
$g_p$ of the metric $g$ at $p$. This formula has two consequences:
first, it shows that the Riemannian distance behaves at the
first-order as the Euclidean distance associated with $\| \cdot
\|_{g_p}$; secondly, the norm $\| \cdot \|_{g_p}$ gives explicit
estimates of $d_R$ near $p$, such as
$$
\frac{1}{C} \| x(q) \|_{g_p} \leq d_R(p,q) \leq C \| x(q) \|_{g_p}.
$$

In sub-Riemannian geometry,  the two properties above hold,
but do not depend on the same function: the first-order behaviour near
$p$ is characterized by the distance $\widehat{d}_p$ defined by a nilpotent
approximation at $p$, whereas explicit local estimates of $d(p,\cdot)$
are given by the pseudo-norm at $p$ $\| \cdot \|_p$ defined
in~\eqref{eq:pseudonorm}. We first present the
latter estimates, often referred to as the ``Ball-Box Theorem'', and
then the
first-order expansion of $d$ in Theorem~\ref{le:dvshatd}.


\begin{theorem}
\label{le:bbox}
The following statement holds if and only if $z_1, \dots, z_n$ are
privileged coordinates at $p$:

\noindent there exist constants $C_p$ and $\eps_p >0$ such that, if
$d(p,q^z) < \eps_p$, then
\begin{eqnarray}
\label{eq:distestim}
\frac{1}{C_p} \|z\|_p \leq d(p,q^z) \leq C_p \| z \|_p
\end{eqnarray}
(as previously, $q^z$ denotes the point near $p$ with coordinates $z$
and  $\| \cdot \|_p$ the pseudo-norm at $p$).
\end{theorem}

\begin{corollary}[Ball-Box Theorem]
Expressed in a given system of privileged coordinates, the sub-Riemannian
balls $B(p,\eps)$ satisfy, for $\eps < \eps_p$,
$$
\mathrm{Box} \big(\frac{1}{C_p} \eps  \big) \subset B(p,\eps) \subset
\mathrm{Box} \big(C_p \eps  \big) ,
$$
where $\mathrm{Box} ( \eps) = [-\eps^{w_1},\eps^{w_1}] \times \cdots
\times [-\eps^{w_n},\eps^{w_n}]$.
\end{corollary}


\begin{remark}
\label{pa:bbox_nonuniform}
The constants $C_p$ and $\eps_p$ depend on the base point $p$.
Around a regular point $p_0$, it is possible to construct  systems
of privileged coordinates depending continuously on the base point
$p$. In this case, the corresponding constants $C_p$ and $\eps_p$
depend continuously on $p$.  This is no longer true at a singular point.
In particular, if $p_0$ is singular, the
estimate~(\ref{eq:distestim}) does not hold uniformly near $p_0$:
we can not choose the constants $C_p$ and $\eps_p$ independently
on $p$ near $p_0$. We will see in section~\ref{se:desing}  uniform
versions of the
Ball-Box Theorem.
\end{remark}
\medskip

The Ball-Box Theorem is stated in different papers, often under the
hypothesis that the point $p$ is regular. To our knowledge, two
valid proofs exist, the ones in~\cite{nag85} and in~\cite{bel96}. The
result also appears without proof in~\cite{gro96} and
in~\cite{ger84}, and with erroneous proofs in~\cite{mit85}
and in~\cite{mon02}.

We present here a proof adapted from the one of Bella\"{\i}che (our is
much simpler because Bella\"{\i}che actually proves a more general
result, namely (\ref{eq:complete_bel})). Basically, the idea is to
compare the distances $d$ and $\widehat{d}$. The main step is
Lemma~\ref{le:x-xhat} below, which is essential in other respects to explain the
role of nilpotent approximations in control theory.\medskip

Fix  a point $p\in M$, and a system of privileged
coordinates at $p$. Through these coordinates we identify a
neighbourhood of $p$ in $M$ with a neighbourhood of 0 in $\R^n$.
As in the preceding subsection, we denote by $\widehat{X}_1, \dots, \widehat{X}_m$ the homogeneous
nilpotent approximation of $X_1, \dots, X_m$ at $p$
(associated with the given system of privileged coordinates) and by $\widehat{d}$ the
induced sub-Riemannian distance on $\R^n$. Recall also that $r=w_n$
denotes the
degree of nonholonomy at $p$.

\begin{lemma}
\label{le:x-xhat}
There exist constants $C$ and $\eps >0$ such that, for any $x_0 \in
\R^n$ and any $t \in \R^+$ with $\tau = \max (\|x_0\|_p,t) < \eps$, we
have
$$
\| x(t) - \widehat{x}(t)  \|_p \leq C \tau t^{1/r},
$$
where $x(\cdot)$ and $\widehat{x}(\cdot)$ are trajectories of the nonholonomic
systems defined respectively by $X_1, \dots, X_m$ and $\widehat{X}_1, \dots, \widehat{X}_m$, starting at
the same point $x_0$, associated with the same control function
$u(\cdot)$, and satisfying $\|u(t)\|=
1$ a.e.
\end{lemma}

\begin{proof}
The first step is to prove that
$\| x (t) \|_p$ and $\| \widehat{x} (t) \|_p \leq \mathit{Cst}\,  \tau$ for small enough
$\tau$, where $\mathit{Cst}$ is a constant. Let us do it for $x(t)$,
the proof being exactly the same for
$\widehat{x}(t)$.

The equation of  the control system associated with
$X_1, \dots, X_m$ is
$$
\dot x_j = \sum_{i=1}^m u_i \big( f_{ij}(x) + r_{ij}(x) \big),
\quad j=1, \dots,n,
$$
where $f_{ij}(x) + r_{ij}(x)$ is of order $\leq w_j -1$ at 0.
Thus, for $j=1, \dots,n$ and $i=1, \dots,m$, $|f_{ij}(x) + r_{ij}(x)| \leq \mathit{Cst}\,  \| x \|_p^{w_j-1}$ when  $\| x \|_p$ is small
enough. Note that, along the trajectory $x(t)$,  $\|x\|_p$ is small when $\tau$ is. Since $\|u(t)\|=
1$ a.e., we get:
\begin{eqnarray}
\label{eq:dotxj}
| \dot x_j | \leq \mathit{Cst}\,  \| x \|_p^{w_j-1}.
\end{eqnarray}

To integrate this inequality, choose an integer $N$ such that all
$N/w_j$ are even integers and set $\| x \|_N= \big( \sum_{i=1}^n
|x_i|^{N/w_i} \big)^{1/N}$. The function $\| x \|_N$ is equivalent
to $\|x\|_p$ in the norm sense, and it is differentiable except at the
origin. Inequality~(\ref{eq:dotxj}) implies
$\frac{d}{dt}\| x \|_N \leq
\mathit{Cst}\, $, and then, by integration,
$$
\| x(t) \|_N \leq \mathit{Cst}\,  t + \| x (0)\|_N \leq \mathit{Cst}\,
\tau.
$$
The functions $\| x \|_N$ and $\|x\|_p$ being equivalent, we
obtain, for a trajectory starting at $x_0$, $\| x (t)\|_p \leq \mathit{Cst}\,
 \tau$ when $\tau$ is small enough. \medskip

The second step is to prove  $|x_j(t) - \widehat{x}_j(t)| \leq \mathit{Cst}\,
\tau^{w_j}t$. The function $x_j - \widehat{x}_j$ satisfies the
differential equation
\begin{eqnarray*}
\dot x_j - \dot{\widehat{x}_j} & = &  \sum_{i=1}^m u_i \big( f_{ij}(x) -
f_{ij}(\widehat{x})  + r_{ij}(x)
\big), \\
& = &  \sum_{i=1}^m u_i \big( \sum_{\{ k \, : \, w_k <w_j\}} (x_k
- \widehat{x}_k) Q_{ijk}(x, \widehat{x}) + r_{ij}(x)
\big),
\end{eqnarray*}
where $Q_{ijk}(x, \widehat{x})$ is a homogeneous polynomial of weighted
degree $w_j-w_k-1$. For $\|x\|_p$ and $\| \widehat{x} \|_p$ small enough, we
have
$$
|r_{ij}(x)| \leq \mathit{Cst}\,  \|x\|_p^{w_j} \quad \hbox{and} \quad
|Q_{ijk}(x, \widehat{x})| \leq \mathit{Cst}\,  (\|x\|_p + \|\widehat{x}\|_p)^{w_j-w_k-1}.
$$
Using the inequalities of the first step, we obtain finally,
for $\tau$ small enough,
\begin{eqnarray}
\label{eq:xj-hatxj}
|\dot x_j (t) - \dot{\widehat{x}}_j(t)| \leq \mathit{Cst}\,   \sum_{\{ k \, : \, w_k
<w_j\}} |x_k (t)- \widehat{x}_k (t)| \tau^{w_j-w_k-1} + \mathit{Cst}\,  \tau^{w_j}.
\end{eqnarray}

This system of inequalities has a triangular form, hence it can be
integrated iteratively. For $w_j=1$,
the inequality is $|\dot
x_j (t) - \dot{\widehat{x}}_j(t)| \leq \mathit{Cst}\,  \tau$, and so $|x_j (t) -
\widehat{x}_j(t)| \leq \mathit{Cst}\,  \tau t$. By induction, let $j > n_1$ and
assume $|x_k (t) - \widehat{x}_k(t)| \leq \mathit{Cst}\,  \tau^{w_k} t$ for $k<j$.
Inequality~(\ref{eq:xj-hatxj}) implies
$$
|\dot x_j (t) - \dot{\widehat{x}}_j(t)| \leq \mathit{Cst}\,  \tau^{w_j-1} t + \mathit{Cst}\,
\tau^{w_j} \leq  \mathit{Cst}\,  \tau^{w_j},
$$
and so $|x_j (t) - \widehat{x}_j(t)| \leq \mathit{Cst}\,  \tau^{w_j} t$.

Finally,
$$
\| x(t) - \widehat{x} (t) \|_p \leq \mathit{Cst}\,  \tau (t^{1/w_1} + \cdots +
t^{1/w_n}) \leq \mathit{Cst}\,  \tau t^{1/r},
$$
which completes the proof of the lemma. 
\end{proof}

\begin{proof}[Proof of Theorem~\ref{le:bbox}]
Observe first that, by definition of order, a system of
coordinates $z$ is privileged if and only if $d(p,q^z) \geq \mathit{Cst}\,
\| z\|_p$. What remains to prove is that, if $z$ are privileged
coordinates, then $d(p,q^z) \leq \mathit{Cst}\,  \| z\|_p$.

We will show that, for $\|x^0\|_p$ small enough,
$$
d(0,x^0) \leq 2 \widehat{d} (0,x^0),
$$
and so $d(0,x^0) \leq \mathit{Cst}\,  \|x^0\|_p$ by
Lemma~\ref{le:nahomo}. This will prove Theorem~\ref{le:bbox}.\medskip

Fix $x^0 \in \R^n$, $\|x^0\|_p < \eps$. Let $\widehat{x}_0 (t)$, $t \in
[0,T_0]$, be a minimizing curve for $\widehat{d}$, having velocity one,
and joining $x^0$ to 0. According to Remark~\ref{re:velocity1}, such a
curve exists, and there exists a control $u_0(\cdot)$
associated with $\widehat{x}_0$ such that $\|u_0(t)\|=1$
a.e. Moreover, $T_0=\widehat{d} (0,x^0)$.

Let
$x_0(t)$, $t \in [0,T_0]$, be the
trajectory of the control system associated with $X_1, \dots, X_m$ starting at
$x^0$ and defined by $u_0(\cdot)$. We have ${\mathrm{length}} \big(x_0
(\cdot)\big) \leq T_0$.
Set $x^1 = x_0(T_0)$. By
Lemma~\ref{le:x-xhat},
$$
\|x^1\|_p = \| x_0(T_0) - \widehat{x}_0 (T_0) \|_p \leq C \tau T_0^{1/r},
$$
where $\tau = \max (\|x^0\|_p, T_0)$. By Lemma~\ref{le:nahomo},
$T_0=\widehat{d} (0,x^0)$ satisfies $T_0 \geq \|x^0\|_p / C'$, so $\tau \leq
C' T_0$, and
$$
\widehat{d} (0,x^1) \leq C'\|x^1\|_p \leq C''  \widehat{d} (0,x^0)^{1+1/r},
$$
with $C''=C'^2 C$.

Choose now $\widehat{x}_1 (t)$, $t \in [0,T_1]$, a minimizing curve for
$\widehat{d}$ of velocity one joining $x^1$ to 0. There exists a
control $u_1(\cdot)$  associated with $\widehat{x}_1$ such that $\|u_1(t)\|=1$
a.e. Let $x_1(t)$,
$t \in [0,T_1]$, be the trajectory of the control system
associated with $X_1, \dots, X_m$ starting at $x^1$ and defined by $u_1(\cdot)$. Set $x^2 = x_1(T_1)$. As
previously, we have ${\mathrm{length}} \big(x_1(\cdot)\big)= \widehat{d} (0,x^1)$ and
$\widehat{d} (0,x^2) \leq C'' \widehat{d} (0,x^1)^{1+1/r}$.

Repeating this construction, we obtain a sequence  of points $x^0, x^1, x^2,
\dots$ such that $\widehat{d} (0,x^{k+1}) \leq C'' \widehat{d}
(0,x^k)^{1+1/r}$, and a sequence of trajectories $x_k(\cdot)$ joining $x^k$
to $x^{k+1}$ of length equal to $\widehat{d} (0,x^k)$.

Taking $\|x^0\|_p$ small enough, we can assume $C''
\widehat{d}(0,x^0)^{1/r} \leq 1/2$. We have then $\widehat{d} (0,x^1)  \leq
\widehat{d}(0,x^0)/2$, \dots, $\widehat{d} (0,x^{k}) \leq \widehat{d}(0,x^0)/2^k$,\dots \
Consequently, $x^k$ tends to 0 as $k\rightarrow +\infty$. Putting end to end the curves $x_k(\cdot)$ gives a trajectory joining $x^0$ to 0 of length
$\widehat{d}(0,x^0) + \widehat{d}(0,x^1) +
\cdots \leq 2 \widehat{d}(0,x^0)$. This implies $d(0,x^0) \leq 2
\widehat{d}(0,x^0)$, and the proof is complete.
\end{proof}
\medskip

Now, the distance $\widehat{d}$ on $\R^n$ induces a distance
$\widehat{d}_p$ on a neighbourhood of $p$ in $M$ by setting
$\widehat{d}_p(q,q') = \widehat{d}(z(q),z(q'))$. This distance gives
the first-order term in the expansion of $d(p,\cdot)$.

\begin{theorem}
\label{le:dvshatd}
On a neighbourhood of $p$ in $M$ there holds
$$
d(p,q) = \widehat{d}_p(p,q) \left(1+
  O\left(\widehat{d}_p(p,q)\right)\right)  .
$$
\end{theorem}

\begin{remark}
 \label{re:bboxhatd}
By Theorem~\ref{le:bbox} and Lemma~\ref{le:nahomo},  when $d(p,q)$ is
small enough we get the estimate
\begin{equation}
  \label{eq:bboxhatd}
  \frac{1}{C} \widehat{d}_p(p,q) \leq d(p,q) \leq C
  \widehat{d}_p(p,q),
\end{equation}
where $C$ is some positive constant. Theorem~\ref{le:dvshatd} essentially
states that
this constant can be chosen arbitrarily close to $1$.
\end{remark}

\begin{proof}
Fix $\delta >0$. We have to prove that there exists $\eps >0$ such
that, if $d(p,q) <
\eps$, then
\begin{equation}
  \label{eq:dvsdhat}
  (1-\delta) \widehat{d}_p(p,q) \leq d(p,q) \leq (1+\delta)
  \widehat{d}_p(p,q).
\end{equation}

Let $q$ be a point in $M$.
 Setting $x^0 =q$
in the proof  of
Theorem~\ref{le:bbox} furnishes a trajectory joining $q$ to 0 which
length is equal to $\sum_{k=0}^\infty \widehat{d}(0,x^k)$, the points
$x^k$ being such that
$\widehat{d} (0,x^{k+1}) \leq C'' \widehat{d}
(0,x^k)^{1+1/r}$.

From~\eqref{eq:bboxhatd}, there exists $\eps >0$ such that $d(p,q) <
\eps$ implies $C'' \widehat{d}(0,x^0)^{1/r} \leq \delta / (1 +
\delta)$. In this case the trajectory from $q$ to $0$ is of length not greater
than $(1 + \delta) \widehat{d}(0,x^0)$ and we have
$$
 d(p,q) \leq (1+\delta)
  \widehat{d}_p(p,q).
$$

To prove the other inequality in~\eqref{eq:dvsdhat}, we use the
same argument but reverse the role of $d$ and $\widehat{d}$. Let $x_0(t)$, $t \in
[0,T_0]$, be a minimizing curve for $d$ of velocity one joining $x^0$ to 0, and let $u_0(\cdot)$ be a control
associated with $x_0(\cdot)$ such that $\|u_0(t)\|=1$
a.e. We have $T_0={d} (0,x^0)$. Let $\widehat{x}_0(t)$, $t \in [0,T_0]$, be the
trajectory of the control system associated with $\widehat{X}_1, \dots,
\widehat{X}_m$ starting at
$x^0$ and defined by the control $u_0(\cdot)$. In particular, ${\mathrm{length}} \big(x_0
(\cdot)\big) \leq T_0$.

Set $x^1 = x_0(T_0)$. By
Lemma~\ref{le:x-xhat},
$$
\|x^1\|_p = \| x_0(T_0) - \widehat{x}_0 (T_0) \|_p \leq C \tau T_0^{1/r},
$$
where $\tau = \max (\|x^0\|_p, T_0)$. Theorem~\ref{le:bbox} implies  $\tau \leq
C_p T_0$, and
$$
{d} (0,x^1) \leq C_p \|x^1\|_p \leq C''  {d} (0,x^0)^{1+1/r},
$$
with $C''=C_p^2 C$.

Repeating this construction gives a trajectory of
$\widehat{X}_1, \dots, \widehat{X}_m$
joining $q$ to $p$  whose
length is equal to $\sum_{k=0}^\infty {d}(0,x^k)$, where
${d} (0,x^{k+1}) \leq C'' {d}
(0,x^k)^{1+1/r}$.

For $d(p,q)$ small enough, we have  $C'' {d}(0,x^0)^{1/r} \leq \delta / (1 -
\delta)$ and the trajectory from $q$ to $0$ is of length $\leq 1/(1 - \delta) {d}(0,x^0)$, which leads to
$$
 \widehat{d}_p(p,q) \leq \frac{1}{(1 - \delta)}
 d(p,q).
$$
This completes the proof.
\end{proof}

\subsection{Approximate motion planning}
Given a control system $(\Sigma)$, the {\em motion planning
problem} is to steer $(\Sigma)$ from an initial point to a final
point. For nonholonomic systems, the exact problem is in general
unsolvable. However methods exist for a particular class of systems,
namely for nilpotent (or nilpotentizable) systems. It is
then of interest to devise approximate motion planning techniques
based on nilpotent approximations. These techniques are Newton
type methods, the nilpotent approximation playing the role of the
usual linearization.

Precisely, consider a nonholonomic control system
$$
(\Sigma) \ : \qquad \dot x = \sum_{i=1}^m u_i X_i(x), \quad x \in
\R^n,
$$
and initial and final points $a$ and $b$ in $\R^n$. Denote by
$\widehat{X}_1, \dots, \widehat{X}_m$ a nilpotent approximation of $X_1, \dots, X_m$ at $b$. The $k$-step of
an approximate motion planning algorithm take the following form
($x^k$ denotes the state of the system, $x^0$ being the initial point $a$):
\begin{enumerate}
  \item compute a control $u(t)$, $t \in [0,T]$, steering the
  control system associated with $\widehat{X}_1, \dots, \widehat{X}_m$ from $x^k$ to $b$;
  \item compute the trajectory $x(\cdot)$ of  $(\Sigma)$ with
  control $u(\cdot)$ starting from  $x^k$;
  \item set $x^{k+1}=x(T)$.
\end{enumerate}

Is this algorithm convergent or,
at least, locally convergent? The answer to the latter question
is positive under an extra hypothesis on the control
 given in point 2 of the algorithm, namely,

\begin{itemize}
  \item[(H)] \ there exists a constant $K$ such that, if $x^k$ and
  $b$ are close enough, then
  $$
\int_0^T \|u(t)\| dt \leq K \widehat{d} (b, x_k).
  $$
\end{itemize}
Note that a control corresponding to a minimizing curve for $\widehat{d}$
satisfies this condition. Other standards methods using Lie groups
(like the one in~\cite{laf91}) or based on the triangular form of
the homogeneous nilpotent approximation also satisfy  this hypothesis.

The local convergence is then proved exactly in the same way than
Theorem~\ref{le:bbox}. We normalize first the control, so that $\|u(t)\| = 1$ a.e. Then from Lemmas~\ref{le:x-xhat}
and~\ref{le:nahomo}, we have $ \widehat{d}(b,x^{k+1}) \leq C'' T^{1+1/r},
$ and using hypothesis (H), we obtain
$$
\widehat{d}(b,x^{k+1}) \leq
C''K^{1+1/r}  \widehat{d} (b, x_k)^{1+1/r}.
$$
If $a$ is close enough to $b$, we have, at each step of the
algorithm, $\widehat{d}(b,x^{k+1}) \leq  \widehat{d} (b, x_k) /2$, which proves
the local convergence of the algorithm. In other words,
%
{\em for each point $b \in M$, there exists a constant
$\eps_b>0$ such that, if $d(a,b) < \eps_b$, then the  approximate
motion planning algorithm steering the system from $a$ to $b$
converges.}\medskip

To obtain a globally convergent algorithm, a natural idea is to
iterate the locally convergent one. This requires the construction
of a finite sequence of intermediate goals $b_0=a, b_1, \dots,
b_N=b$ such that $d(b_{i-1},b_i) < \eps_{b_i}$. However the
constant $\eps_b$ depends on $b$ and, as already noticed for
Theorem~\ref{le:bbox}, it is not possible to have a uniform
nonzero constant near singular points. Thus this method may provide  a
globally convergent algorithm only when every point is regular.


\section{Tangent structure to Carnot-Carath\'{e}odory
  spaces}
\label{se:tgtstruct}
Consider a  manifold $M$ endowed with a
sub-Riemannian distance $d$ on $M$. The so-defined
metric space $(M,d)$ is called a \emph{Carnot-Carath\'{e}odory
  space}.
  The notion of first-order approximation introduced in the previous section has a  metric interpretation and will allow us to describe the local
structure of a Carnot-Carath\'{e}odory
  space.

\subsection{Metric tangent space}
\label{se:tangentcone}
In describing the tangent space to a manifold, we
essentially look at smaller and smaller neighbourhoods of a
given point, the manifold being fixed. Equivalently, we can
 look at a fixed neighbourhood, but expanding the
manifold. As noticed by Gromov, this idea can be used to define a notion of tangent space for a general metric
space.

If $\textsf{X}$ is a metric space with distance $d$, we define $\lambda
\textsf{X}$, for $\lambda
>0$, to be  the metric space with same underlying set
 as $\textsf{X}$ and distance $\lambda d$. A {\em pointed metric
space} $(\textsf{X},x)$ is a metric space with a distinguished point $x$.

Loosely speaking, a  metric tangent space  to the metric space $\textsf{X}$
at $x$ is a pointed metric space $(C_x\textsf{X},y)$ such that
$$
(C_x\textsf{X},y) = \lim_{\lambda \rightarrow + \infty} (\lambda \textsf{X},x).
$$
Of course, for this definition to make sense, we have to define the
limit of pointed metric spaces.

Let us first define the Gromov-Hausdorff distance between metric
spaces. Recall that, in a metric space $\textsf{X}$, the Hausdorff
distance $\hbox{H-dist} (A,B)$ between two subsets $A$ and $B$ of
$\textsf{X}$ is the infimum of $\rho>0$ such that any point of $A$ is
within a distance $\rho$ of $B$ and any point of $B$ is within a
distance $\rho$ of $A$. The Gromov-Hausdorff distance
$\hbox{GH-dist} (X,Y)$ between two metric spaces $\textsf{X}$ and $\textsf{Y}$
is the infimum of Hausdorff distances $\hbox{H-dist}
(i(\textsf{X}),j(\textsf{Y}))$ over all metric spaces $\textsf{Z}$ and all isometric
embeddings $i : \textsf{X} \rightarrow \textsf{Z}$, $j: \textsf{Y} \rightarrow \textsf{Z}$.

Thanks to Gromov-Hausdorff distance, one can define the notion of
limit of a sequence of pointed metric spaces: $(\textsf{X}_n,x_n)$
converge to $(\textsf{X},x)$ if, for any positive $r$,
$$
\hbox{GH-dist} \big( B^{\textsf{X}_n}(x_n,r), B^\textsf{X}(x,r) \big)
\rightarrow 0  \quad \hbox{as} \ n \rightarrow + \infty
$$
where $B^\textsf{Y}(y,r)$ is considered as a metric space, endowed with the distance of $Y$.
Note that all pointed metric spaces isometric to $(\textsf{X},x)$ are also limit of $(\textsf{X}_n,x_n)$.
However the limit is unique up to an isometry provided the closed balls around the distinguished point are compact~\cite[Sect.\ 7.4]{Burago2001}.

Finally, one says that $(\textsf{X}_\lambda,x_\lambda)$
converge to $(\textsf{X},x)$  when $\lambda \to \infty$ if, for every sequence $\lambda_n$,
$(\textsf{X}_{\lambda_n},x_{\lambda_n})$ converge to $(\textsf{X},x)$.

\begin{definition}
A pointed metric space $(C_x\textsf{X},y)$ is a {\em metric tangent space}  to the metric space $\textsf{X}$
at $x$ if $(\lambda \textsf{X},x)$ converge to $(C_x\textsf{X},y)$ as $\lambda \rightarrow + \infty$.
If it exists, it is unique up to an isometry provided the closed balls around $x$ in $(\lambda \textsf{X},x)$  are compact.
\end{definition}

For a Riemannian metric space $(M,d_R)$ induced by a Riemannian metric $g$ on a manifold $M$,
metric tangent spaces at a
point $p$ exist and are isometric to the Euclidean space $(T_pM, g_p)$, that
is, the standard tangent space endowed with the scalar product
defined by the quadratic form $g_p$.

For a Carnot-Carath\'{e}odory space $(M,d)$,
the metric tangent space is given by the
nilpotent approximation.

\begin{theorem}
\label{th:tangentcone}
A Carnot-Carath\'{e}odory space $(M,d)$ admits metric tangent spaces $(C_pM,y)$ at
every point $p\in M$. The space $C_pM$ is itself a Carnot-Carath\'{e}odory space isometric
to $(\R^n,\widehat{d})$, where $\widehat{d}$ is the sub-Riemannian distance associated with a
homogeneous nilpotent approximation at $p$.
\end{theorem}

This theorem, due to Bella\"{\i}che, is a consequence of a strong version of
Theorem~\ref{le:bbox}: for $q$ and $q'$ in a neighbourhood of $p$,
\begin{eqnarray}
\label{eq:complete_bel}
|d(q,q')- \widehat{d} (q,q')| \leq  \mathit{Cst} \, \widehat{d}(p,q) d(q,q')^{1/r}.
\end{eqnarray}
In these notes, we present neither the proof of this result, nor the one of Theorem~\ref{th:tangentcone}, and we refer the reader to~\cite{bel96}.

\begin{remark}
Recall that $\widehat{d}$ is not intrinsic to the frame
$(X_1,\dots,X_m)$. Thus Theorem~\ref{th:tangentcone} does not provide an intrinsic characterization of the
metric tangent space. Such characterizations exist for sub-Riemannian manifolds $(M,D,g_R)$ in~\cite{mar00} and
\cite{fal03}, and the latter could easily be adapted to the case of a sub-Riemannian geometry associated with a nonholonomic system. However these characterizations are intrinsic to the differentiable manifold $M$ equipped with the sub-Riemannian structure $(D,g_R)$, or to $M$ equipped with the frame $(X_1,\dots,X_m)$,  not to the metric space $(M,d)$. To our knowledge, the problem of finding a definition of the metric tangent space $C_pM$ depending only on the Carnot-Carath\'{e}odory space $(M,d)$ is still open.
\end{remark}

The question we want to address now is:
what is the algebraic structure of $C_pM$? Of course $C_pM$ is not a
linear space in general: for instance, $\widehat{d}$ is homogeneous of
degree 1 but with respect to dilations $\delta_t$ but not with respect to the
usual Euclidean dilations. We will see that $C_pM$ has a natural
structure of group, or at least of quotient of groups.


Denote by $G_p$ the group generated by the diffeomorphisms $\exp(
t \widehat{X}_i)$ acting on the left on $\R^n$. Since
$\mathfrak{g}_p= \mathit{Lie} (\widehat{X}_1, \dots, \widehat{X}_m)$
is a nilpotent Lie algebra, $G_p = \exp (\mathfrak{g}_p)$ is a simply
connected group, having $\mathfrak{g}_p$ as its Lie algebra.

This Lie algebra $\mathfrak{g}_p$ splits into homogeneous components
$$
\mathfrak{g}_p = \mathfrak{g}^{-1} \oplus \cdots \oplus \mathfrak{g}^{-r},
$$
where $\mathfrak{g}^{-s}$ is the set of homogeneous vector fields of degree
$-s$, and so $\mathfrak{g}_p$ is a graded Lie algebra. The first component $\mathfrak{g}^{-1}=
\mathrm{span} \langle \widehat{X}_1, \dots, \widehat{X}_m \rangle$ generates $\mathfrak{g}_p$ as a Lie algebra. All
these properties imply that $G_p$ is what we  call a Carnot group.

\begin{definition}
A {\em Carnot group} is a simply connected Lie group, such that the associated Lie algebra is graded, nilpotent, and generated by its first component.
\end{definition}

Note that the
dilations $\delta_t$ act on $\mathfrak{g}_p$ as a multiplication by $t^{-s}$ on
$\mathfrak{g}^{-s}$. This action extends to $G_p$ by the exponential mapping.

\begin{example}[Heisenberg group]
\label{ex:heisgroup}
The simplest non Abelian Carnot group is the \emph{Heisenberg group} $\mathbb{H}^3$ which is the connected and simply connected Lie group whose Lie algebra satisfies
$$
\mathfrak{g} = \mathfrak{g}^{-1} \oplus  \mathfrak{g}^{-2}, \qquad \hbox{with } \dim \mathfrak{g}^{-1} = 2.
$$
As a consequence, $\dim \mathbb{H}^3 = 3$. Choosing a basis $X$, $Y$, $Z = [X,Y]$ of $\mathfrak{g}$, we define coordinates on $\mathbb{H}^3$ by the exponential mapping
$$
(x,y,z) \mapsto \exp (x X + y Y + z Z).
$$
By the Campbell-Hausdorff formula (see Section~\ref{se:chforflows}),  the law group on $\mathbb{H}^3$ in these coordinates is
$$
(x,y,z) \cdot (x',y',z') = (x+x', y+y', z+z'+\frac{1}{2} (xy'-x'y)),
$$
which is homogeneous with respect to the dilation $\delta_t (x,y,z) = (tx,ty,t^2z)$.

Finally, denote by $X_1, X_2$ the left-invariant vector fields on $\mathbb{H}^3$ whose values at the identity are respectively $X$ and $Y$. In coordinates $(x,y,z)$, these vector fields write as
$$
X_1  = {\partial_x} - \frac{y}{2}  {\partial_z} \quad \mbox{ and } \quad
X_2 = {\partial_y} + \frac{x}{2}  {\partial_z},
$$
which are the vector fields of what we have called the Heisenberg case in examples~\ref{ex:heis} and~\ref{ex:heis2}.
\end{example}

%

Let $\widehat{\xi}_1, \dots, \widehat{\xi}_m$ be the right-invariant vector fields
on $G_p$ such that $\widehat{\xi}_i ( \mathrm{id}) = \widehat{X}_i$, where $\mathrm{id}$ is the identity of $G_p$. Equivalently,
$$
\widehat{\xi}_i (g) = \frac{d}{dt} \big[ \exp (t \widehat{X}_i) g \big] \big|_{t=0}.
$$
With $(\widehat{\xi}_1, \dots, \widehat{\xi}_m)$ is associated a
right-invariant sub-Riemannian metric and a sub-Riemannian distance
$d_{G_p}$ on $G_p$.

The action of $G_p$ on $\R^n$ is smooth and transitive. Indeed, for
every $x \in
\R^n$, the orbit of $x$ under the
    action of $G_p$ is the set
$$
\left\{ \exp(t_{i_1} \widehat{X}_{i_1}) \circ \cdots \circ \exp(t_{i_k}
\widehat{X}_{i_k})(x) \ : \ k\in \N, \ t_{i_j} \in \R, \ i_j \in \{ 1,
\dots, m\} \right\}.
$$
By Chow-Rashevsky's theorem (or more precisely from
    Remark~\ref{re:orbit}), this set is
the whole $\R^n$ since $( \widehat{X}_1, \dots, \widehat{X}_m )$
    satisfies Chow's Condition on $\R^n$ (Lemma~\ref{le:nahomo}).

To understand the algebraic structure of $C_pM$ we will use the following standard result on transitive action
of Lie groups (see for instance~\cite[Th.\ 9.24]{lee03}).

\begin{theorem}
\label{th:lee}
Let $G$ be a Lie group acting on the left smoothly and transitively on
a manifold
$M$. Let $q \in M$ and $H$ be the isotropy
subgroup of $q$ which is defined by $H=\{ g \in G \ : \ g \cdot q = q \}$. Then
$H$ is a closed subgroup of $G$,
the left coset space $G/H$ is a manifold of dimension $\dim G - \dim
H$, and  the map $F : G/H \to M$  defined by $F(gH) = g \cdot q$  is
an equivariant diffeomorphism.
\end{theorem}

Let $H_p$ be the isotropy subgroup of $0 \in \R^n$ under the action of $G_p$. According to
Theorem~\ref{th:lee}, the  map $\phi_p : G_p \rightarrow \R^n$,
$\phi_p(g)=g(0)$,  induces
a diffeomorphism
$$
\psi_p : G_p / H_p \rightarrow \R^n, \qquad \psi_p(gH_p) = g(0).
$$

Observe that $H_p$ is invariant under dilations, since $\delta_t g
(\delta_t x) = \delta_t (g(x))$. Hence $H_p$ is connected and
simply connected, and so $H_p = \exp (\mathfrak{h}_p)$, where
$\mathfrak{h}_p$ is
the Lie sub-algebra of $\mathfrak{g}_p$ containing the vector fields
vanishing at $0$,
$$
\mathfrak{h}_p = \{ Z \in \mathfrak{g}_p \ : \ Z(0)=0 \}.
$$
As $\mathfrak{g}_p$,  $\mathfrak{h}_p$ is  invariant under dilations and splits into
homogeneous components.


Now, the elements $\widehat{X}_1, \dots, \widehat{X}_m$ of $\mathfrak{g}_p$ act on the left on
$G_p / H_p$ with the notation $\overline{\xi}_1,
\dots, \overline{\xi}_m$,
$$
\overline{\xi}_i (gH_p) = \frac{d}{dt} \big[ \exp (t \widehat{X}_i) g H_p\big] \big|_{t=0}.
$$
These vector fields define a sub-Riemannian metric and a
sub-Riemannian distance $\overline{d}$ on $G_p / H_p$. We clearly
have ${\psi_p}_*\overline{\xi}_i = \widehat{X}_i$, so $\psi_p$ maps
the sub-Riemannian
metric on $G_p / H_p$ associated with $(\overline{\xi}_1,
\dots, \overline{\xi}_m)$ to the one on $\R^n$ associated with $( \widehat{X}_1,
\dots, \widehat{X}_m )$. We summarize this construction by the following result.

\begin{theorem}
\label{th:algstruc}
The metric tangent space $C_pM$ and
$(\R^n,\widehat{d})$ are isometric to the coset
space $G_p / H_p$ endowed with the sub-Riemannian distance $\overline{d}$.
\end{theorem}


\begin{example}[Gru\v{s}in plane]
\label{ex:grusin}
Consider the vector fields $X_{1}=\partial_x$ and $X_2 = x \partial_y$
on $\R^2$. The Carnot-Carath\'{e}odory space defined by these vector
fields is called the \emph{Gru\v{s}in plane}.

The only nonzero bracket is $X_{12} = \partial_y$. Thus, at $p=0$, the
weights are $(1,2)$, and $(x,y)$ are privileged coordinates. Since
$X_1$ and $X_2$ are homogeneous with respect to this system of
coordinates, we have $\widehat{X}_1 = X_1$ and $\widehat{X}_2=
X_2$. The Lie algebra they generate is
$$
\mathfrak{g}_0 = \mathrm{span} (X_1, X_2, X_{12}) 
$$
which is of dimension 3, and the group $\exp (\mathfrak{g}_0)$ is
actually the \emph{Heisenberg group} $\mathbb{H}^3$ (see
example~\ref{ex:heisgroup}).
The Lie sub-algebra $\mathfrak{h}_0$ of $\mathfrak{g}_0$ containing
the vector fields
vanishing at $0$ is
$$
\mathfrak{h}_0 = \mathrm{span} (X_2), 
$$
which is one-dimensional. Thus the Gru\v{s}in plane is isometric to $\mathbb{H}^3 / \exp(\mathfrak{h}_0)$ endowed with the distance $\overline{d}$.
\end{example}

\begin{example}[Martinet case]
\label{ex:martinet2}
Consider the Martinet case, defined on $\R^{3}$ by
$$
X_1  = {\partial_x}  \quad \mbox{ and } \quad
X_2 = {\partial_y} + \frac{x^2}{2}  {\partial_z}.
$$
As noticed in example~\ref{ex:martinet}, at $p=0$, the coordinates $(x,y,z)$ are privileged and, by homogeneity, $\widehat{X}_1 = X_1$ and $\widehat{X}_2= X_2$. Moreover the only nonzero bracket are $X_{12} = x\partial_z$ and $X_{112} = \partial_z$. Thus,
$$
\mathfrak{g}_0 = \mathrm{span} (X_1, X_2, X_{12},X_{112} ) 
$$
which is of dimension 4. The group $\exp (\mathfrak{g}_0)$ is called the \emph{Engel group}, and is denoted by $\mathbb{E}^4$.
The Lie sub-algebra $\mathfrak{h}_0$ of $\mathfrak{g}_0$ containing the vector fields
vanishing at $0$ is
$$
\mathfrak{h}_0 = \mathrm{span} (X_{12}). 
$$
\end{example}
\medskip

When the point $p$ is regular, Theorem~\ref{th:algstruc} can be refined thanks to the following result.

\begin{lemma}
If $p$ is a regular point, then $\dim G_p=n$.
\end{lemma}

\begin{proof}
Let $X_{I_1}, \dots, X_{I_n}$ be an adapted frame at $p$.
Due to the regularity of $p$, $X_{I_1}, \dots, X_{I_n}$ is also an adapted frame
near $p$, so any bracket $X_J$ can be written as
$$
X_J (z) = \sum_{\{i \, : \, |I_i| \leq |J| \}} a_i(z) X_{I_i} (z),
$$
where  each $a_i$ is a function of order $\geq |I_i| - |J|$. Taking the homogeneous terms of degree $- |J|$ in this expression,
we obtain
$$
\widehat{X}_J (z) = \sum_{ \{i \, : \, |I_i| = |J|\}} a_i(0) \widehat{X}_{I_i}
(z),
$$
and so $\widehat{X}_J \in \mathrm{span} \langle \widehat{X}_{I_1},
\dots, \widehat{X}_{I_n}
\rangle$. Thus $\widehat{X}_{I_1}, \dots, \widehat{X}_{I_n}$ is a basis of
$\mathfrak{g}_p$, and so $\dim G_p=n$. 
\end{proof}

As a consequence $H_p$ is of dimension zero. Since $H_p$ is invariant under dilations, $H_p=\{ \mathrm{id} \}$, and hence the mapping $\phi_p : G_p \rightarrow \R^n$,
$\phi_p(g)=g(0)$, is a diffeomorphism. Moreover ${\phi_p}_*\widehat{\xi}_i=\widehat{X}_i$, which
implies that $\phi_p$
maps the sub-Riemannian
metric on $G_p$  associated with $(\widehat{\xi}_1, \dots, \widehat{\xi}_m)$ to
the one on $\R^n$ associated with $( \widehat{X}_1, \dots, \widehat{X}_m )$. This
gives the following result\footnote{This result appeared first in~\cite{mit85}, but with an erroneous proof. The presentation given here is inspired from the one of~\cite{bel96}.}.

\begin{proposition}
When $p$ is a regular point, the metric tangent space $C_pM$ and
the Carnot-Carath\'{e}odory space $(\R^n,\widehat{d})$ are isometric to
the Carnot group $G_p$ endowed with the right-invariant sub-Riemannian
distance $d_{G_p}$.
\end{proposition}

Thus  Carnot groups have the same role in sub-Riemannian geometry as
Euclidean spaces have in
Riemannian geometry. For this reason they are sometimes referred to as
``non Abelian linear spaces'': the internal operation -- addition --
is
replaced by the law group and the external operation -- multiplication by
a real number --  by the dilations.
Note that, when $G_p$ is
Abelian (i.e.\  commutative) then $G_p$ has a linear structure and the
sub-Riemannian metric on $G_p$ is a Euclidean metric.


\begin{example}[unicycle]
\label{ex:unicycle3}
In the case of the distance $d$ associated with the unicycle (examples~\ref{ex:unicycle} and~\ref{ex:unicycle2}), the growth vector is $(2,3)$ at every point. Hence every point $p \in \R^2 \times S^1$ is regular, and the Lie algebra generated by the nilpotent approximation satisfies
$$
\mathfrak{g}_p = \mathfrak{g}^{-1} \oplus  \mathfrak{g}^{-2}, \qquad
\hbox{with } \dim \mathfrak{g}^{-1} = 2.
$$
As a consequence, $G_p = \mathbb{H}^3$ (see
example~\ref{ex:heisgroup}), and so  the metric tangent space to
$(\R^2 \times S^1,d)$ at every point $p$ has the structure of the
Heisenberg group.
\end{example}

\subsection{Desingularization and uniform distance estimate}
\label{se:desing}
We have already highlighted in Remark~\ref{pa:bbox_nonuniform} that singular points may cause difficulties, in particular because of the loss of uniformity of distance estimates. Therefore it is necessary to  study carefully the behaviour of the distance at such points. We proceed as it is usual for singularities, that is,  we consider a
singularity as the projection of a regular object.
To this aim we exploit the algebraic structure of the metric tangent space, which provides a good
way of lifting and projecting Carnot-Carath\'{e}odory spaces.

Let us begin with
nilpotent approximations.
We keep the notations and definitions of the preceding subsection. At
a singular point $p$, we have the following
diagram,
$$
\begin{array}{ccc}
  (G_p, d_{G_p}) &  &  \\[2mm]
  \hbox{\scriptsize $\pi$ } \hbox{\Large $\downarrow$} &
  \hbox{\scriptsize $\phi_p$} \hbox{\Large  $\searrow$}  & \\[1mm]
  (G_p/H_p, \overline{d}) & \stackrel{\psi_p}{\widetilde{\longrightarrow}} &
  (\R^n,\widehat{d})
\end{array}
$$
Since the sub-Riemannian metric on $G_p$ is a right-invariant,
every point in the space $(G_p, d_{G_p})$
is regular.

\begin{definition}
A Carnot-Carath\'{e}odory space $(M,d)$ is said to be \emph{equiregular} if every point in $M$ is regular.
\end{definition}
Thus $(\R^n,\widehat{d})$ is the
projection of the  equiregular space $(G_p, d_{G_p})$.
Recall now that $\widehat{\xi}_1, \dots, \widehat{\xi}_m$ (resp. $\overline{\xi}_1, \dots,
\overline{\xi}_m$) are  mapped to $\widehat{X}_1, \dots, \widehat{X}_m$ by
$\phi_p$ (resp. $\psi_p$). Working in a system of coordinates, we
identify $G_p/H_p$ with $\R^n$ and $\overline{\xi}_i$ with $\widehat{X}_i$. These
coordinates on $\R^n \simeq G_p/H_p$, denoted by $x$, induce coordinates $(x,z)
\in \R^N$ on $G_p$ for which  we have
\begin{eqnarray}
\label{eq:bij}
\widehat{\xi}_i(x,z) = \widehat{X}_i (x) + \sum_{j=n+1}^N b_{ij} (x,z)
\partial_{z_j}.
\end{eqnarray}

Let $(x(\cdot),z(\cdot))$ be a trajectory of the nonholonomic system in $G_p$ defined by $\widehat{\xi}_1, \dots, \widehat{\xi}_m$. Then, for every control $u(\cdot)$ associated with the trajectory,
$$
(\dot x (t), \dot z (t)) = \sum_{i=1}^m
 u_i(t) \widehat{\xi}_i (x,z),
$$
It follows from~\eqref{eq:bij} that $x(\cdot)$ is a trajectory in $\R^n$ of the system defined by $\widehat{\xi}_1, \dots, \widehat{\xi}_m$, which is associated with the same controls $u(\cdot)$, so that
$$
{\mathrm{length}} \big(x(\cdot)\big) = {\mathrm{length}} \big((x,z)(\cdot) \big).
$$
Thus $\widehat{d}$ can be obtained
from the sub-Riemannian distance $d_{G_p}$ in $G_p$ by
$$
\widehat{d} (q_1, q_2) = \inf_{\widetilde{q}_2 \, \in \, q_2H_p} d_{G_p}(\widetilde{q}_1,
\widetilde{q}_2), \qquad \hbox{for any } \widetilde{q}_1 \in q_1H_p,
$$
or, equivalently, $B^{\widehat{d}} (q_1,\eps) = \phi_p \big(
B^{d_{G_p}}(\widetilde{q}_1, \eps) \big)$.\medskip

We will use this idea to desingularize the original space $(M,d)$.
Choose for $x$ privileged coordinates at $p$, so that
$$
X_i(x) = \widehat{X}_i (x) + R_i(x) \qquad \hbox{with } \mathrm{ord}_pR_i \geq 0.
$$
Set $\widetilde{M}= M \times \R^{N-n}$, and in local coordinates
$(x,z)$ on $\widetilde{M}$, define vector fields on a neighbourhood
of $(p,0)$ by
$$
\xi_i (x,z) = X_i (x) + \sum_{j=n+1}^N b_{ij} (x,z)
\partial_{z_j},
$$
with the same functions $b_{ij}$ as in~(\ref{eq:bij}). Such vector
fields are called a \emph{lifting} of the vector fields
$X_1,\dots,X_m$: denoting by $\pi: \widetilde{M} \to M$ the canonical
projection, we have $X_i = \pi_*\xi_i$ for $i=1, \dots,m$, that is,
$X_i$ is the projection of $\xi_i$.

We define in this way a nonholonomic system on an open set
$\widetilde{U} \subset \widetilde{M}$ whose nilpotent
approximation at $(p,0)$ is $(\widehat{\xi}_1,
\dots, \widehat{\xi}_m)$, by construction. Unfortunately, $(p,0)$ can be itself a singular
point. Indeed, a point can be singular for a system and
regular for the nilpotent approximation taken at this point.

\begin{example}
Take the vector fields $X_1=\partial_{x_1}$,
$X_2 =
\partial_{x_2} + x_1 \partial_{x_3} + x_1^2 \partial_{x_4}$ and
$X_3 = \partial_{x_5} + x_1^{100} \partial_{x_4}$ on $\R^5$. The origin 0 is
a singular point. However the nilpotent approximation at 0 is
$\widehat{X}_1=X_1$, $\widehat{X}_2=X_2$, $\widehat{X}_3=\partial_{x_5}$, for which $0$ is
not singular.
\end{example}

To avoid this difficulty, we take a group bigger than $G_p$,
namely
 the free nilpotent group $N_r$ of step $r$ with $m$ generators. $N_r$ is a
 Carnot group and its
 Lie algebra $\mathfrak{n}_r$ is the free nilpotent Lie algebra of step $r$
 with $m$ generators. The given of $m$ generators $\alpha_1, \dots,
 \alpha_m$ of
 $\mathfrak{n}_r$ define on $N_r$ a right-invariant sub-Riemannian distance
  $d_N$.

 The group $N_r$ can be thought as a group of diffeomorphisms, and
 so it defines a left action on $\R^n$. Denoting by $J$ the isotropy
 subgroup of 0 for this action, we obtain that $(\R^n, \widehat{d})$ is
 isometric to $N_r/J$ endowed with the restriction of the distance $d_N$.

 Reasoning as above, we are able to lift locally the
 vector fields $X_1, \dots, X_m$ on $M$  to vector fields on $M
 \times \R^{\widetilde{n}-n}$,
 $\widetilde{n}=\dim N_r$, having $\alpha_1, \dots, \alpha_m$ for
 nilpotent  approximation at $(p,0)$.
 Moreover $(p,0)$ is a regular point for the associated
 nonholonomic system in $M \times \R^{\widetilde{n}-n}$ since $N_r$ is
 free up to step  $r$.
We obtain in this way a result of desingularization.

\begin{lemma}
Let $p$ be a point in $M$, $r$ the degree of nonholonomy at $p$,
$\widetilde{n}=\widetilde{n} (m,r)$ the dimension of the free Lie
algebra of step $r$
with $m$ generators, and
$\widetilde{M}= M \times \R^{\widetilde{n}-n}$.
Then there exist a neighbourhood $\widetilde{U} \subset
\widetilde{M}$ of $(p,0)$, a neighbourhood $U \subset M$ of $p$
with $U \times \{0 \} \subset \widetilde{U}$, local coordinates
$(x,z)$ on $\widetilde{U}$, and smooth vector fields on
$\widetilde{U}$,
\begin{equation}
\label{eq:projxi}
\xi (x,z) = X_i (x) + \sum_{j=n+1}^N b_{ij} (x,z)
\partial_{z_j} ,
\end{equation}
such that:
\begin{itemize}
  \item $\xi_1, \dots, \xi_m$  satisfy Chow's Condition and have  $r$
    for degree of nonholonomy
  everywhere (so the Lie algebra they generate is free up to step $r$);
  \item every $\widetilde{q}$ in $\widetilde{U}$ is regular;
  \item denoting by $\pi : \widetilde{M} \rightarrow M$ the canonical
  projection, and by $\widetilde{d}$ the sub-Riemannian distance defined by $\xi_1, \dots, \xi_m$ on
  $\widetilde{U}$, we have $\pi_*\xi_i = X_i$, and for $q \in U$ and
  $\eps>0$ small
  enough,
  $$
B(q,\eps) = \pi \Big( B^{\widetilde{d}}   \big( (q,0), \eps \big)
\Big),
  $$
  or, equivalently,
  $$
d(q_1,q_2) = \inf_{\widetilde{q}_2 \, \in \, \pi^{-1} (q_2)} \widetilde{d}
\big( (q_1,0), \widetilde{q}_2
\big).
  $$
\end{itemize}
 \end{lemma}

\begin{remark}
The lemma still holds if we replace $r$ by any integer greater than
the degree of nonholonomy at $p$.
\end{remark}

Thus any Carnot-Carath\'{e}odory space $(M,d)$ is locally the projection of an
equiregular
Carnot-Carath\'{e}odory space $(\widetilde{M},\widetilde{d})$. This
projection preserves the trajectories, the
minimizers, and the distance.


\begin{example}[Martinet case]
\label{ex:martinet3}
Consider the vector fields of the Martinet case (see
example~\ref{ex:martinet}), defined on $\R^{3}$ by:
$$
X_1  = {\partial_x}  \quad \mbox{ and } \quad
X_2 = {\partial_y} + \frac{x^2}{2}  {\partial_z}.
$$
Let $\pi: \R^4 \to \R^3$ be the projection
with respect to the last coordinates, $\pi(x,y,z,w) = (x,y,z)$. Then $X_1$
and $X_2$ are the projections of the vector fields defining the
Engel group $\mathbb{E}^4$ (see example~\ref{ex:martinet2}),
$$
\xi_1  = {\partial_x}  \quad \mbox{ and } \quad
\xi_2 = {\partial_y} + \frac{x^2}{2}  {\partial_z} + x {\partial_w},
$$
that is $\pi_*\xi_i = X_i$. Thus, for every pair of points $q_1,q_2 \in
\R^3$,
$$
d_{\mathrm{Mart}}(q_1,q_2) = \inf_{w \in \R} d_{\mathbb{E}^4}
\big( (q_1,0), (q_2,w)
\big),
$$
where $d_{\mathrm{Mart}}$ and $d_{\mathbb{E}^4}$ are the
sub-Riemannian distance in respectively the Martinet space and the
Engel group.
\end{example}

\begin{example}[Gru\v{s}in plane]
\label{ex:grusin2}
Consider the vector fields
$$
X_{1}=\partial_x, \qquad X_2 = x \partial_y,
$$
on $\R^2$, which define the Gru\v{s}in plane (see
example~\ref{ex:grusin}). Let $\pi: \R^3 \to \R^2$ be the projection
with respect to the last coordinates, $\pi(x,y,z) = (x,y)$. Then
$X_1=\pi_*\xi_1$
and $X_2=\pi_*\xi_2$, where
$$
\xi_1  = {\partial_x}  \quad \mbox{ and } \quad
\xi_2 = {\partial_z} + x  {\partial_z},
$$
are, up to a change of coordinates, the vector fields defining the
Heisenberg case (see example~\ref{ex:heis}).
\end{example}

\paragraph{Application: uniform Ball-Box theorem}\

The key feature of  regular points is uniformity:
\begin{itemize}
  \item uniformity of the flag~(\ref{eq:flag});
  \item uniformity w.r.t.\ $p$ of the convergence $(\lambda (M,d), p)
  \rightarrow C_pM$ (as explained by Bella\"{\i}che~\cite[Sect.\ 8]{bel96},
   this uniformity is responsible for the group
   structure of the metric tangent space);
  \item uniformity of distance estimates (see Remark~\ref{pa:bbox_nonuniform}).
\end{itemize}
In particular the last property is essential to compute Hausdorff
dimensions (see Section~\ref{se:hausdim}) or to prove the global convergence of approximate motion
planning algorithms. Recall what we mean by uniformity in this
context: in a neighbourhood of a regular point $p_0$, we can
construct privileged coordinates depending continuously on the
base point $p$ and such that the distance
estimate~(\ref{eq:distestim}) holds with $C_p$ and $\eps_p$
independent of $p$.

As already noticed, all these uniformity properties are lost at
singular points. However, using the desingularization of a sub-Riemannian
manifold, we are able to give a uniform version of distance
estimates.\medskip

Let $\Omega \subset M$ be a compact set. We denote by $r_{\max}$
the maximum of  degrees of nonholonomy at points in  $\Omega$. As noticed in Section~\ref{se:privcoor}, $r_{\max}$ is finite. We assume
that $M$ is an oriented manifold, and we choose a volume form $\omega$ on $M$.

Let $\mathfrak{X}$ be the set of $n$-tuples $\mathbf{X}=(X_{I_1}, \dots, X_{I_n})$ of
brackets of length $|I_i| \leq r_{\max}$. Since $r_{\max}$ is finite, $\mathfrak{X}$ is a finite subset of $\mathit{Lie}(X_1,\dots,X_m)^n$.
 Given $q \in \Omega$ and $\eps
>0$, we define a function $f_{q,\eps}: \mathfrak{X} \to \R$ by
$$
f_{q,\eps}( \mathbf{X}) = \left| \omega_q \big( X_{I_1}(q) \eps^{|I_1|},
\dots, X_{I_n}(q) \eps^{|I_n|} \big) \right|.
$$
We say that $\mathbf{X} \in \mathfrak{X}$ is  an {\em  adapted frame at
$(q,\eps)$} if it achieves the maximum of $f_{q,\eps}$ on $\mathfrak{X}$.

The values at $q$ of an adapted frame at $(q,\eps)$ clearly
form a basis of $T_qM$. Moreover, $q$ being fixed, the adapted
frames at $(q,\eps)$ are adapted frames at $q$ for
$\eps$ small enough.

\begin{theorem}[Uniform Ball-Box theorem~\cite{jea01b}]
\label{th:unifbbox}
There exist positive constants $K$ and $\eps_0$ such that, for
$q \in \Omega$, $\eps < \eps_0$, and any adapted
frame $\mathbf{X}$ at $(q,\eps)$, there holds
$$
\mathrm{Box}_\mathbf{X} (q, \frac{1}{K}\eps) \subset B(q,\eps) \subset
\mathrm{Box}_\mathbf{X} (q, K \eps),
$$
where $\mathrm{Box}_\mathbf{X} (q, \eps) = \{ \exp(x_1 X_{I_1}) \circ
\cdots \circ \exp(x_n X_{I_n})(q)  : \, |x_i| \leq \eps^{|I_i|}, \,
1 \leq i \leq n \}$.
\end{theorem}
We do not give the proof of this theorem in these notes, we refer the reader to~\cite{jea01b}.

Of course, when the point $q$ is fixed, the estimate above is equivalent to the one
of the Ball-Box theorem for $\eps$ smaller than some $\eps_1(q)>0$.
However, the main difference is that here $\eps_0$ does not depend on $q$, whereas in the Ball-Box theorem $\eps_1=\eps_1(q)$ can be infinitely close to $0$ as $q$ varies.

As a direct consequence of the Uniform Ball-Box theorem,
 we have an estimate of the volume of a small sub-Riemannian ball.
 Let $\mathrm{vol}_\omega$ be the measure on $M$ associated with $\omega$, \emph{i.e.}\/ for any measurable subset $A \subset M$ we set $\mathrm{vol}_\omega(A) = \int_A \omega$.

\begin{corollary}
\label{th:volballs}
There exist positive constants $K$ and $\eps_0$ such that, for all
$q \in \Omega$ and $\eps < \eps_0$,
$$
\frac{1}{K} \max_{\mathbf{X}} f_{q,\eps}( \mathbf{X}) \leq \mathrm{vol}_\omega (B(q,\eps))
\leq K  \max_{\mathbf{X}} f_{q,\eps}( \mathbf{X}),
$$
the maximum of $f_{q,\eps}( \mathbf{X}) = \left| \omega_q \big( X_{I_1}(q) \eps^{|I_1|},
\dots, X_{I_n}(q) \eps^{|I_n|} \big) \right|$ being taken over all families
$\mathbf{X}=(X_{I_1}, \dots, X_{I_n})$ of
brackets of length $|I_i| \leq r_{\max}$.

If moreover all points in $\Omega$ are regular, then  for all
$q \in \Omega$ and $\eps < \eps_0$,
\begin{equation}
\label{eq:volreg}
    \frac{1}{K} \eps^Q \leq \mathrm{vol}_\omega (B(q,\eps))
\leq K  \eps^Q,
\end{equation}
where $Q = \sum_{i=1}^n w_i(q)$ does not depend on $q$.
\end{corollary}


\subsection{Hausdorff dimension}
\label{se:hausdim}

Consider a metric space $(M,d)$ and denote by $\diam S$ the diameter of a set $S \subset M$.
Let $k \geq 0$ be a real number. For every subset $A \subset M$, we
define the
\emph{$k$-dimensional Hausdorff measure} $\hh^k$ of $A$  as $\hh^k(A)
= \lim_{\eps \to 0^+} \hh^k_\eps(A)$, where
$$
\hh^k_\eps(A) = \inf \left\{ \sum_{i=1}^\infty  \left(\diam S_i\right)^k
\, : \, A \subset \bigcup_{i=1}^\infty S_i, \ S_i \hbox{ closed set}, \  \diam S_i \leq \eps \right\},
$$
and the \emph{$k$-dimensional spherical Hausdorff measure} $\ss^k$  of
$A$  as $\ss^k(A)
= \lim_{\eps \to 0^+} \ss^k_\eps(A)$, where
$$
\ss^k_\eps(A) = \inf \left\{ \sum_{i=1}^\infty  \left(\diam
S_i\right)^k \, : \, A \subset \bigcup_{i=1}^\infty S_i, \ S_i \hbox{ is
  a ball}, \ \diam
S_i \leq \eps  \right\}.
$$

In the Euclidean space $\R^n$,  $k$-dimensional Hausdorff measures
are often
defined as $2^{-k}\alpha(k)\hh^k$ and $2^{-k}\alpha(k)\ss^k$, where
$\alpha(k)$ is defined from the
usual gamma function as $\alpha(k) = \Gamma (\frac{1}{2})^k / \Gamma
(\frac{k}{2}+1)$. This normalization
factor is necessary for the $n$-dimensional Hausdorff measure and
the Lebesgue measure to coincide on $\R^n$.

For a given set $A \subset M$, $\hh^k(A)$ is a decreasing function
of $k$, infinite when $k$ is smaller than a certain value, and zero
when $k$ is greater than this value. We call \emph{Hausdorff
dimension} of $A$ the real number
$$
\dim_\hh A = \sup \{ k \, : \, \hh^k (A) = \infty \} = \inf \{ k \,
: \, \hh^k (A) = 0 \}.
$$
Note that $\hh^k \leq \ss^k \leq 2^k \hh^k$, so the Hausdorff
dimension can be defined equivalently from Hausdorff or spherical
Hausdorff measures.\medskip

There exist only few results on Hausdorff measures in sub-Riemannian geometry, except for specific cases~\cite{Agrachev2010,Ghezzi2012}. The most general result is the following one.

\begin{theorem}
\label{th:dimhausreg}
Let $(M,d)$ be an equiregular Carnot-Carath\'{e}odory space and $p$ a point in $M$. Then the
Hausdorff dimension of a small enough ball $B(p,r)$ is $\dim_\hh B(p,r) = Q$, where
$$
Q= \sum_{i=1}^n w_i(p) = \sum_{i\geq 1} i \left( \dim \Delta^i(p) -  \dim \Delta^{i-1}(p) \right)
$$
does not depend on $p$. Moreover  $\hh^{Q} (B(p,r))$ is finite.
\end{theorem}

\begin{proof}
Fix a volume form $\omega$ on $B(p,r)$ (it is possible for a  small enough $r$), and denote by $\mathrm{vol}_\omega$ the associated measure.
It results from Corollary~\ref{th:volballs} that, for $q \in B(p,r)$ and $\eps$ small enough,
\begin{equation}
\label{eq:volequireg}
\frac{1}{K} \eps^Q \leq \mathrm{vol}_\omega (B(q,\eps))
\leq K  \eps^Q.
\end{equation}

Define $N_\eps$ to be the maximal number of disjoints balls of radius $\eps$ included in $B(p,r)$, and consider such a family
$B(q_i,\eps)$, $i=,\dots,N_\eps$, of disjoints balls. By~\eqref{eq:volequireg},
$$
\frac{1}{K}  \eps^Q N_\eps \leq \mathrm{vol}_\omega (B(p,r)) \quad \Rightarrow \quad N_\eps  \leq K \eps^{-Q} \mathrm{vol}_\omega (B(p,r)).
$$
On the other hand  the union $\bigcup_i B(q_i,2\eps)$ covers $B(p,r)$, and by Theorem~\ref{th:unifbbox}  every ball $B(q_i,2\eps)$ is of diameter $\geq  \frac{4}{K} \eps$ if $\eps$ is small enough. This implies
$$
\ss^Q (B(p,r)) \leq \liminf_{\eps \to 0} N_\eps \left(\frac{4  \eps}{K}\right)^Q < \infty.
$$
Therefore $\dim_\hh B(p,r) \leq Q$.

Conversely, let $\bigcup_i B(q_i,r_i)$ be a covering of $B(p,r)$ with
balls of diameter not greater than $\eps$. If $\eps$ is small enough,
every $r_i$ is smaller than $\eps_0$ and there holds
$$
\mathrm{vol}_\omega (B(p,r)) \leq \sum_i \mathrm{vol}_\omega (B(q_i,r_i)) \leq K
\sum_i r_i^Q.
$$
As a consequence, we have $\ss^Q (B(p,r)) \geq \mathrm{vol}_\omega (B(p,r))
/K$, which in turn implies $\dim_\hh B(p,r) \geq Q$. This ends the
proof.
\end{proof}

When $(M,d)$ is not equiregular, the Hausdorff dimension of balls
centered at singular points behaves in a different way. Let us show it
on an example.

Consider the Martinet space (see Example~\ref{ex:martinet}), that is,
$\R^3$ endowed with the sub-Riemannian distance associated with the
vector fields
$$
X_1  = {\partial_x}  \quad \mbox{ and } \quad
X_2 = {\partial_y} + \frac{x^2}{2}  {\partial_z}.
$$
A point $q=(x,y,z)$ is regular if $x\neq 0$ and in this case $\sum_i
w_i(q)=4$, otherwise it is singular and $\sum_i w_i(q)=5$.

\begin{lemma}
    \label{le:dimhausmartinet}
Let $p$ a point in the Martinet space.
\begin{itemize}
  \item If $p$ is regular, then $\dim_\hh B(p,r) = 4$, and $\hh^4
    (B(p,r))$ is finite.
  \item If $p$ is singular, then $\dim_\hh B(p,r) = 4$, but $\hh^4
    (B(p,r))$ is not finite.
\end{itemize}
\end{lemma}

\begin{proof}
When $p$ is regular, the result is a direct consequence of Theorem~\ref{th:dimhausreg}. Let us consider a singular point $p$ and a radius $r>0$. Since regular points form an open set, $B(p,r)$ contains small balls centered at regular points, and thus $\dim_\hh B(p,r) \geq 4$. Moreover, it results from Corollary~\ref{th:volballs} that, for $q=(x,y,z)$ close enough from $p$ and for $\eps >0$ small enough,

\begin{equation}
\label{eq:volmart}
\frac{1}{K} \eps^4 \max (|x|, \eps)  \leq \mathrm{vol}_\omega (B(q,\eps))
\leq K  \eps^4 \max (|x|, \eps).
\end{equation}

We proceed as in the proof of Theorem~\ref{th:dimhausreg}. Define
$N_\eps$ to be the maximal number of disjoints balls of radius $\eps$
included in $B(p,r)$, and consider such a family
$B(q_i,\eps)$, $i=,\dots,N_\eps$, of disjoints balls, with
$q_i=(x_i,y_i,z_i)$. Notice that the first coordinate $x$ is of
nonholonomic order $\leq 1$ at any point; this implies that there
exists a constant $K'>0$ such that
$$
B(q_i,\eps) \subset B(p,r) \cap \{ q=(x,y,z) \ : \ |x-x_i| \leq K' \eps \}.
$$
As a consequence, for an integer $k$, every ball $B(q_i,\eps)$ such
that $(k-1) \eps \leq |x_i| < k \eps$ is included in  the set $B(p,r)
\cap \{ q=(x,y,z) \ : \ |x| \in ((k-1-K') \eps, (k+K')\eps] \}$. The
volume of the latter set is smaller than $K'' \eps$, where $K''$ is a
constant (depending neither on $k$ nor $\eps$). Then it results
from~\eqref{eq:volmart} that
$$
K N_\eps(k) k \eps^5 \leq K'' \eps,
$$
where $N_\eps(k)$ is the number of points $q_i$ such that $(k-1) \eps
\leq |x_i| < k \eps$. The Ball-Box Theorem implies that $N_\eps(k)=0$
when $k>K'r/\eps$, and hence
$$
N_\eps = \sum_{k=1}^{\lceil K'r/\eps\rceil} N_\eps(k) \leq
\frac{\mathrm{const}}{\eps^4} \sum_{k=1}^{\lceil K'r/\eps\rceil}
\frac{1}{k} \leq \frac{\mathrm{const}}{\eps^4} \log \left(\frac{1}{\eps}\right),
$$
where $\lceil t \rceil$ denotes the integer part of a number $t$.
Now the union $\bigcup_i B(q_i,2\eps)$ covers $B(p,r)$ and every ball
$B(q_i,2\eps)$ is of diameter $\geq  \frac{4}{K} \eps$ if $\eps$ is
small enough. This implies that, for any real number $s > 4$,
$$
\ss^s (B(p,r)) \leq \lim_{\eps \to 0}\left(\frac{4  \eps}{K}\right)^s
N_\eps \leq \lim_{\eps \to 0} \mathrm{const}\,  \eps^{s-4} \log
\left(\frac{1}{\eps}\right) =0.
$$
Consequently $\dim_\hh B(p,r) \leq 4$, and hence $\dim_\hh B(p,r) =4$,
since the converse inequality holds.\medskip

We are left to show that $\hh^4 (B(p,r))$, or equivalently $\ss^4
(B(p,r))$, is not finite.
Let $\bigcup_i B(q_i,r_i)$ be a covering of $B(p,r)$ with balls of
diameter not greater than $\eps$.
For an integer $k\geq 1$, denote by $I_k$ the set of indices such that
$\bigcup_{i\in I_k} B(q_i,r_i)$ is a covering of the set $B(p,r) \cap
\{ q=(x,y,z) \ : \ |x| \in ((k-1) \eps, k\eps] \}$. Thus
$$
\sum_{i\in I_k} \mathrm{vol}_\omega (B(q_i,r_i)) \geq \mathrm{const} \, \eps.
$$
On the other hand $i\in I_k$ implies $\mathrm{vol}_\omega (B(q_i,r_i)) \leq
\mathrm{const} \ r_i^4 k \eps$, and so
$$
\sum_{i\in I_k} r_i^4 \geq \frac{\mathrm{const}}{k}.
$$
Summing up over $k$, we obtain, for a small enough $\eps$,
$$
\sum_{i} r_i^4 \geq \mathrm{const}\,  \log \left(\frac{1}{\eps}\right).
$$
As a consequence, $\ss^4_\eps (B(p,r)) \geq \mathrm{const}\,  \log
(\frac{1}{\eps})$, and so $\ss^4_\eps (B(p,r))=\infty$. This ends the
proof.
\end{proof}


\appendix

\newpage

\noindent {\Large \textbf{Appendix}}
\addcontentsline{toc}{section}{
    \textbf{Appendix}}

\section{Flows of vector fields}
\label{se:flowsvf}
This section is dedicated to the proof of Campbell-Hausdorff type
formulas for flows of vector fields. The result in
Section~\ref{se:chforflows} has been used in Section~\ref{se:chow}, the one in
Section~\ref{se:pushforward} will be necessary for the next section.\medskip

Let $U$ be an open subset of $\R^n$ and $VF(U)$ the set of smooth
vector fields on $U$. Given a vector field $X \in VF(U)$, we denote   its flow by
$\exp(tX)$.

\subsection{Campbell-Hausdorff formula for flows}
\label{se:chforflows}


We will need in this section the Campbell-Hausdorff formula which we recall
briefly here (for a more detailed presentation see for
instance \cite[Ch.\ II]{bou72}).  Let $x$ and $y$ be two non
commutative indeterminates, and $[x,y]=xy-yx$ their
commutator, also denoted by $[x,y]=(\ad x)y$. The
length of an iterated commutator $(\ad x_1) \cdots (\ad x_{k-1}) x_k$,
where each $x_1,\dots, x_k$ equals $x$ or $y$, is defined to be
the number of occurrences $k$ of $x$ and $y$. Define
also $e^x$ and $e^y$ to be  the series $\sum_{k \geq 0}
\frac{x^{k}}{k!}$ and $\sum_{k \geq 0} \frac{y^{k}}{k!}$. Then we have
$e^x e^y = e^{H(x,y)}$ in the sense of formal power series, where
\begin{eqnarray}
\label{eq:chseries}
H(x,y) = x+y+ \frac{1}{2} [x,y] + R(x,y), 
\end{eqnarray}
and $R(x,y)$ is a series whose terms are linear combination of iterated commutators of $x$ and $y$ of length greater than 2. For an integer $N$ we denote by $H_N(x,y)$ the partial sum of $H(x,y)$ containing only iterated commutators of length not greater than $N$. In particular, $H_1=x+y$ and $H_2= x+y+ \frac{1}{2} [x,y]$.\medskip

Consider now two vector fields $X,Y \in VF(U)$. Given $t \in \R$ and an integer $N$, $H_N(tY,tX)$ is a smooth vector field on $U$ which writes as $\sum_{i=1}^N t^i Y_i$, where the vector fields $Y_1,\dots, Y_N$ belong to the Lie algebra generated by $X$ and $Y$.

\begin{lemma}
\label{le:taylor_CH}
Let $p\in M$. There exist positive constants $\delta$ and $C$ such that $|t| < \delta$ implies
$$
\left\| \exp(tX) \circ \exp (tY) (p) - \exp(H_N(tY,tX))(p)\right\| \leq C |t|^{N+1}.
$$
\end{lemma}

\begin{proof}
Set $\psi(t)=\exp(tX) \circ \exp (tY) (p)$, which is a function defined and $C^\infty$ in a neighbourhood of $0 \in \R$, and let $(x_1,\dots,x_n)$ be a system of local coordinates on a neighbourhood of $p$ in $U$. We will compute the Taylor expansion of every component $x_i(\psi(t))$, for $i=1, \dots,n$. To do this, we introduce  the function $\phi(t,s)= \exp(tX) \circ \exp (sY) (p)$, so that $\psi(t) = \phi(t,t)$, and we compute the partial derivatives of $x_i \circ \phi$ at $0 \in \R^2$. We have:
$$
\frac{\partial x_i \circ \phi}{\partial t} (t,s) = \frac{d}{dt}\left[ x_i \circ \exp(tX) \right] \left(\exp(sY)(p)\right) = Xx_i(\phi(t,s)),
$$
where $Xx_i$ denotes the Lie derivative of $x_i$ along $X$. Repeating this computation, we obtain  for any integer $k$,
$$
\frac{\partial^k x_i \circ \phi}{\partial t^k} (t,s) =  X^kx_i(\phi(t,s)).
$$
In the same way, we have:
\begin{multline*}
\frac{\partial^{k+l} x_i \circ \phi}{\partial s^l \partial t^k} (0,0) = \frac{\partial^{l}}{\partial s^l} \frac{\partial^{k} x_i \circ \phi}{ \partial t^k} (0,s)\big|_{s=0} \\
= \frac{\partial^{l}}{\partial s^l} \left[ X^kx_i(\exp (sY) (p))\right]\big|_{s=0} = Y^l X^kx_i(p).
\end{multline*}
We then deduce that the formal Taylor series of $x_i(\psi(t)) = x_i(\phi(t,t))$ at $0$ is
$$
\sum_{k,l\geq 0} \frac{t^{k+l}}{k! l!} Y^l X^kx_i(p) = \left[ \sum_{l\geq 0} \frac{t^l}{l!} Y^l \right] \left[ \sum_{k\geq 0} \frac{t^k}{k!} X^k \right] x_i(p),
$$
where $X$ and $Y$ are considered as derivation operators.
From the Campbell-Hausdorff formula, the product of the formal series $e^{tY} = \sum_{l\geq 0} \frac{t^l}{l!} Y^l$ with $e^{tX} = \sum_{k\geq 0} \frac{t^k}{k!} X^k $ is equal to the series $e^{H(tY,tX)}$. As a consequence, the Taylor expansion of $x_i(\psi(t))$ up to degree $N$ is given by the terms of degree $\leq N$ in the series $e^{H(tY,tX)}x_i (p)$, which coincide with the terms of degree $\leq N$ in the series $e^{H_N(tY,tX)}x_i (p)$.

On the other hand, it results from Lemma~\ref{le:dlexpY(t)} below that $e^{H_N(tY,tX)}x_i (p)$ is the Taylor series at $0$ of the function $t \mapsto x_i \circ \exp(H_N(tY,tX))(p)$. Thus
$$
x_i \circ \psi(t) - x_i \circ \exp(H_N(tY,tX))(p)= O(|t|^{N+1})
$$
 for every coordinate $x_i$, and the lemma follows.
\end{proof}

\begin{lemma}
\label{le:dlexpY(t)}
Let $Y_1,\dots, Y_{\ell}$ be vector fields on $U$,  $f:U \to \R$ a smooth function, and $p \in U$. The formal Taylor series at $0 \in \R^\ell$ of the function $(z_1,\dots,z_\ell) \mapsto f(\exp(\sum_i z_i Y_i)(p))$ is given by
$$
\sum_{k\geq 0} \frac{1}{k!} (\sum_i z_i Y_i)^k f (p) = e^{\sum_i z_i Y_i}f(p).
$$
The formal Taylor series at $0 \in \R$ of the function $t \mapsto f(\exp(Y(t))(p))$, where $Y(t)=\sum_{i=1}^\ell t^i Y_i$,  is given by
\begin{equation}
\label{eq:dlexpY(t)}
\sum_{k\geq 0} \frac{1}{k!} Y(t)^k f (p)= e^{Y(t)}f(p).
\end{equation}
\end{lemma}

\begin{proof}
The second statement is obviously a  consequence of the first one, so it
is sufficient to prove the latter. We introduce the functions $g(z)=
f(\exp(\sum_i z_i Y_i)(p))$ and $G(t,z)=g(tz)$ which are well-defined and smooth on a
neighbourhood of $0$ in $\R^\ell$,  respectively $\R \times
\R^\ell$. We are looking for the Taylor series of $g$ at $0$.

Since $G(t,z)= f(\exp(t \sum_i z_i Y_i)(p))$, we have, for any integer $k \geq 0$,
$$
\frac{\partial^k G}{\partial t^k}(0,z) = (\sum_i z_i Y_i)^k f (p).
$$
On the other hand $G(t,z)=g(tz)$, and hence the previous derivative can also be computed as
$$
\frac{\partial^k G}{\partial t^k}(0,z) = \sum_{\alpha_1 + \cdots + \alpha_\ell=k} \frac{k!}{\alpha_1! \cdots \alpha_\ell !} z_1^{\alpha_1} \cdots z_\ell^{\alpha_\ell} \frac{\partial^k g}{\partial z_1^{\alpha_1} \cdots \partial z_\ell^{\alpha_\ell}}(0) .
$$
Combining both expressions, we obtain
$$
\sum_{\alpha_1 + \cdots + \alpha_\ell=k} \frac{z_1^{\alpha_1} \cdots z_\ell^{\alpha_\ell}}{\alpha_1! \cdots \alpha_\ell !}  \frac{\partial^k g}{\partial z_1^{\alpha_1} \cdots \partial z_\ell^{\alpha_\ell}}(0) =\frac{1}{k!} \left(\sum_i z_i Y_i\right)^k f (p),
$$
and the lemma follows.
\end{proof}

Lemma~\ref{le:taylor_CH} can be extended in two ways. First, since the vector fields and their flows are smooth on $U$, the estimate holds uniformly with respect to $p$. Second, by Lemma~\ref{le:dlexpY(t)}, the vector fields $tX$ and $tY$ may be replaced by the one-parameter families of vector fields $X(t)= tX_1 + \cdots + t^{k} X_k$ and $Y(t)=t Y_1 + \cdots + t^\ell Y_\ell$, where $X_1,\dots,X_{k}$ and $Y_1,\dots, Y_{\ell}$ are vector fields on $U$. As an example,  $H_N(tY,tX)$ is of this form. To summarize, a slight change in the proof of Lemma~\ref{le:taylor_CH} actually shows the following result.

\begin{corollary}
\label{le:corochform}
Let $K \subset U$ be a compact. There exist two positive constants $\delta,C$ such that, if $p \in K$ and $|t| < \delta$, then:
$$
\left\| \exp(X(t))\circ \exp(Y(t))(p) - \exp(H_N(Y(t),Y(t)))(p)\right\| \leq C |t|^{N+1}.
$$
\end{corollary}\medskip

We are now in a position to prove formula~\eqref{eq:phiI}, that we  used in the proof of Lemma~\ref{le:Apopen}.
Let $X_1,\dots,X_m$ be $m$ elements of $VF(U)$. For every multi-index $I \in \{1,\dots,m\}^k$, $k\in \N$, we define the local diffeomorphisms $\phi^I_t$ on $U$ by  induction on the length $|I|$ of $I$. Let $\phi^i_t = \exp(tX_i)$ and set, if $I=iJ$,
$$
\phi_t^{iJ}
=\phi_{-t}^J \circ \phi_{-t}^i
\circ \phi_t^J \circ\phi_t^i.
$$

\begin{proposition}
\label{le:quasich}
Let $K \subset U$ be a compact and $I$ a multi-index. There exist two positive constants $\delta,C$ such that, if $p \in K$ and $|t| < \delta$, then
\begin{eqnarray}
\label{eq:phIt}
  \left\| \phi^I_t (p) - p - t^{|I|} X_{I} (p) \right\| & \leq & C |t|^{|I|+1}.
\end{eqnarray}
\end{proposition}

\begin{proof}
For $\delta >0$ small enough, the mapping $(p,t) \mapsto \phi^I_t(p)$ is defined and $C^\infty$ on $K \times (-\delta,\delta)$. As a consequence, we are reduced to prove~\eqref{eq:phIt} for a fixed $p\in K$.

When $|I|=1$, $\phi^I_t (p)=\exp(tX_i)(p)$ for some $i\in \{1,\dots,m\}$, which is equal to $p + t X_i(p) +O(|t|^2)$.

Now, let $I$ be a multi-index and $N > |I|$ an integer. Corollary~\ref{le:corochform} implies that $\phi_t^{I}(p)= \exp(H_N^{I}(t))(p) + O(t^{N+1})$ where the series $H_N^I(t)$ is defined by induction: if $I=iJ$, then
$$
H_N^{iJ}(t)=H_N( H_N(tX_i,H_N^J(t)),H_N(-tX_i,-H_N^J(t))).
$$
Applying~\eqref{eq:chseries} iteratively, we can write
$H_N^I(t) = t^{|I|} X_I + t^{|I|+1} R_I(t)$, the latter term being a one-parameter vector field. As a consequence,
$$
\phi_t^{I}(p)= p + t^{|I|} X_I (p)+ \hbox{terms of degree greater than }  {|I|+1},
$$
which completes the proof.
\end{proof}


\subsection{Push-forward formula}
\label{se:pushforward}

Given two vector fields $X, Y \in VF(U)$, we write $(\ad X)Y$ for $[X,Y]$, $(\ad X)^2Y$ for $(\ad X)((\ad X)Y)$, etc.

\begin{proposition}
\label{th:quasich}
Let $K \subset U$ be a compact, $N$ a positive integer, and $X$, $Y$, $Y_1,\dots,Y_\ell$  vector fields in $VF(U)$. There exist two positive constants $\delta,C$ such that, if $p \in K$, $t \in \R$ and $z\in \R^\ell$ satisfy $|t| < \delta$ and $\|z\| < \delta$, then
\begin{eqnarray}
 & \left\| \exp(tY)_*X (p) - \sum_{k = 0}^N \Frac{t^{k}}{k!} (\ad Y)^{k}X (p)  \right\|  \leq  C |t|^{N+1}, & \nonumber \\[2mm]
  &  \left\| \exp(\sum_{i=1}^\ell z_i Y_i)_*X (p) - \sum_{k = 0}^N \Frac{1}{k!} \left(\ad \sum_{i=1}^\ell z_i Y_i\right)^{k}X (p)  \right\| \leq  C \|z\|^{N+1}, &\nonumber
\end{eqnarray}
where $\exp(tY)_*X= d (\exp(tY)) \circ  X \circ \exp(-tY)$ denotes the
push-forward of the vector field $X$ by the diffeomorphism $\exp(tY)$.
\end{proposition}

\begin{proof}
Let us begin with the first inequality. Set $\phi_p(t) = \exp(tY)_*X (p)$. For $\delta >0$ small enough, the mapping $(p,t) \mapsto \phi_p(t)$ is defined and $C^\infty$ on $K \times (-\delta,\delta)$. As a consequence, there exists a constant $C>0$ such that, for every $p\in K$ and $|t| < \delta$, we have
\begin{eqnarray}
  \left\| \phi_p(t) - \sum_{k = 0}^N \Frac{t^{k}}{k!} \frac{d^k \phi_p}{d t^k}(0)  \right\| & \leq & C |t|^{N+1}. \nonumber
\end{eqnarray}
It remains to prove that $\frac{d^k \phi_p}{d t^k}(0) = (\ad Y)^{k}X (p)$ for any integer $k$. Note first that $\phi_p(0)=X(p)$, and that
$$
\frac{d \phi_p}{d t}(0) = \frac{d }{d t} \left[ \exp(tY)_*X  \right]\big|_{t=0}(p)
$$
is by definition equal to $-L_YX(p)$, where $L_YX$ is the Lie derivative of $X$ along $Y$ (see for instance~\cite{boo86}). Since $L_YX(p)=(\ad Y) X(p)$, the cases $k=0$ and $k=1$ are done.

We need now to compute $\frac{d \phi_p}{d t}(t)$ at $t\neq 0$. Let us write $\phi_p(t+s)$ as $\exp(tY)_*\exp(sY)_*X (p)$. We  have
\begin{multline*}
\frac{d \phi_p}{d t}(t) = \frac{d \phi_p(t+s)}{d s}\mid_{s=0}= \exp(tY)_*\frac{d }{d s}\left[ \exp(sY)_*X \right]\big|_{s=0}(p) \\
=\exp(tY)_* ((\ad Y) X)(p).
\end{multline*}
This derivative has the same form as $\phi_p(t)$, $X$ being replaced by $(\ad Y)X$. Iterating the  argument above, we obtain by induction $\frac{d^k \phi_p}{d t^k}(0) = (\ad Y)^{k}X (p)$, and the first inequality of the proposition is proved.\medskip

As for the second inequality, the same reasoning applies and we only need to compute the partial derivatives at $0 \in \R^\ell$ of the function $\widetilde{\phi}(z)=\exp(\sum_{i} z_i Y_i)_*X (p)$. This can be done as in the proof of Lemma~\ref{le:dlexpY(t)}. The proposition follows.
\end{proof}


\section{Different systems of privileged coordinates}
\label{se:proofpriv}
This appendix is devoted to the proof that the examples of coordinates
introduced in Section~\ref{se:privcoor} are actually privileged
coordinates.

\subsection{Canonical coordinates of the second kind}
\label{se:2ndkind}
Let $p\in M$ and $Y_1,\dots,Y_n$ an adapted frame at $p$
(see~\eqref{eq:adaptedframe}, page~\pageref{pa:adaptedframe}).  The map
$$
\phi  : (z_1,\dots,z_n) \mapsto \exp(z_n Y_n)\circ \cdots \circ \exp(z_1
Y_1)(p)
$$
 is a local diffeomorphism near $0 \in \R^n$ and its inverse defines some
 coordinates called
 \emph{canonical coordinates of the second kind near $p$}.

The following result is due to Hermes~\cite{her91}.

\begin{lemma}
\label{ccan}
Canonical coordinates of the second
kind are privileged  at $p$.
\end{lemma}
For sake of simplicity, we will write the compositions of maps as
 products; for instance, we write
$$
\phi(z)= \exp(z_1 Y_1)
 \cdots   \exp(z_n Y_n)(p).
$$

\begin{proof}
First, let us  recall that $\phi$ is a local diffeomorphism at $z=0$
because its differential at $0$ is an isomorphism. This
results from
$$
\frac{\partial \phi}{\partial
z_i}(0)=
\frac{d}{dt}\left(\phi(0,
\dots,t,\dots,0)\right)\big|_{t=0}=\frac{d}{dt}\left(
\exp(tY_i)(p)\right)\big|_{t=0}=Y_i(p),
$$
for $i=1,\dots,n$. This computation also reads as $\phi_*
\frac{\partial}{\partial z_i} (p) = Y_i(p)$, which implies  $Y_i
z_i(p)=1$ (as in Section~\ref{se:nhorders}, $Y_iz_i$ denotes the Lie
derivative of the
function $z_i$ along the vector field $Y_i$). Hence the order of $z_i$ at $p$
is not greater than $w_i$.

It remains to show that the order of $z_i$ at $p$
is at least $w_i$ for each $i=1,\dots, n$. This is a direct consequence
of the following assertion. \medskip

\noindent \textbf{Claim.} \emph{Let $X$ be one of the vector fields
  $X_1, \dots,
  X_m$. Then, for $i=1,\dots, n$, the Taylor expansion at $z=0$
  of the function $a_i(z) = Xz_i \left(\phi(z)\right)$ is a sum of homogeneous
  polynomials in the coordinates $z$
of weighted degree  $\geq w_i-1$.}\smallskip

From the very definition of $a_i(z)$, we have
\begin{eqnarray}
\label{eq:ai(z)}
 X(\phi(z)) & = & \sum_{i=1}^n a_i(z) \frac{\partial \phi}{\partial
   z_i} (z).
\end{eqnarray}
Given $z$, let $\varphi$ be the diffeomorphism defined on a neighbourhood
of  $p$ by $\varphi(q)  = \exp(z_1 Y_1)
 \cdots   \exp(z_n Y_n)(q)$. In particular, $\varphi(p)=\phi(z)$. In order to obtain an equality in $T_pM$, we  apply the isomorphism $(d \varphi_p)^{-1}$ to both sides of~\eqref{eq:ai(z)}, and we get
\begin{eqnarray*}
(\varphi^{-1})_* X (p) & = & \sum_{i=1}^n a_i(z) \, (\varphi^{-1})_*\frac{\partial }{\partial
   z_i} (p).
\end{eqnarray*}
This equality is of the form $W= \sum_{i=1}^n a_i V_i$, where the
vectors $W=W(z)$ and $V_i=V_i(z)$, $i=1, \dots,n$, belong to
$T_pM$. If we denote by $b=b(z)\in\R^n$ the coordinates of $W$ in the
basis $(Y_1(p), \dots, Y_n(p))$ of $T_pM$, and by $P=P(z)$ the
$(n\times n)$-matrix of the coordinates of $V_1,\dots,V_n$ in the
same basis, then the vector $a(z)=(a_1(z), \dots,a_n(z))$  appears as
the solution of   $Pa=b$.

Note first that  $P(0)$ equals the identity matrix
$\mathbb{I}$. Therefore both matrices  $P(z)$ and
$P(z)^{-1}$ are equal to  $\mathbb{I} + $homogeneous terms of positive
degree. Hence  the Taylor expansion of $a_i(z)$ and the one of
$b_i(z)$ have the same homogeneous terms of lower degree.

On the other hand, since $\varphi^{-1} = \exp(-z_n Y_n) \cdots \exp(-z_1 Y_1)$, we have
\begin{eqnarray*}
W(z) & = &  \exp(-z_n
Y_n)_* \cdots \exp(- z_1 Y_1)_* X  (p).
\end{eqnarray*}
Let us choose an integer $N$ bigger than all the weights $w_i$, and apply Proposition~\ref{th:quasich} to $\exp(- z_1 Y_1)_* X$, then to $\exp(- z_2 Y_2)_* (\ad Y_1)^{l_1}X$, and so on,
\begin{eqnarray*}
W(z)  & = &  \exp(-z_n
Y_n)_* \cdots \exp(- z_2 Y_2)_*\sum_{l_1= 0}^N \frac{(-z_1)^{l_1}}{l_1!} (\ad Y_1)^{l_1}X (p) \\
& & \qquad \qquad \qquad \qquad \qquad \qquad \qquad \qquad \qquad \qquad  + O(|z_1|^{N+1})\\
& \vdots &  \\
& = &  \sum_{l_1,\dots,l_{n} = 0}^N \frac{(-z_1)^{l_1}}{l_1!} \cdots
\frac{(-z_n)^{l_n}}{l_n!} \:
(\ad Y_n)^{l_n}\cdots (\ad Y_1)^{l_1}X (p) + O(|z|^{N+1}).
\end{eqnarray*}
Hence every coordinate $b_i(z)$ of $W(z)$ satisfies
\begin{eqnarray}
\label{eq:bi}
b_i(z)  & = &  \sum_{l_1,\dots,l_{n} = 0}^N
 \frac{(-z_1)^{l_1}}{l_1!} \cdots
\frac{(-z_n)^{l_n}}{l_n!} \beta_i^l + O(|z|^{N+1}),
\end{eqnarray}
 $\beta_i^l$  being the
$i$th coordinate in the basis $(Y_1(p), \dots, Y_n(p))$ of the vector $(\ad Y_n)^{l_n}\cdots (\ad Y_1)^{l_1}X (p)$. The
latter vector belongs to $\Delta^w(p)$,
where $w=1+l_1w_1 +\cdots+l_nw_n$ (recall that $X \in \Delta^1$ and
$Y_i \in \Delta^{w_i}$). Since $(Y_1,
\dots, Y_n)$  is an adapted frame at $p$, $\beta_i^l$ is zero when
$1+l_1w_1 +\cdots+l_nw_n < w_i$. It follows that  $b_i(z)$ -- and then
 $a_i(z)$ -- contains only homogeneous terms of weighted degree
 greater than or equal to $w_i -1$. This ends the proofs of both the claim
 and  the lemma.
\end{proof}


\subsection{Canonical coordinates of the first kind}
\label{se:1stkind}
Let $p\in M$ and $Y_1,\dots,Y_n$ an adapted frame at $p$.  The map
$$
\widetilde{\phi}  : (z_1,\dots,z_n) \mapsto \exp(z_1 Y_1 + \cdots + z_n
Y_n)(p)
$$
 is a local diffeomorphism near $0 \in \R^n$ since its differential at $0$ is an isomorphism. This
results from
$$
\frac{\partial \widetilde{\phi}}{\partial
z_i}(0)=
\frac{d}{dt}\left(\widetilde{\phi}(0,
\dots,t,\dots,0)\right)\big|_{t=0}=\frac{d}{dt}\left(
\exp(tY_i)(p)\right)\big|_{t=0}=Y_i(p),
$$
for $i=1,\dots,n$. The inverse of $\widetilde{\phi}$ defines some local
 coordinates near $p$ called
 \emph{canonical coordinates of the first kind}.

\begin{lemma}
\label{le:ccan1}
Canonical coordinates of the first
kind are privileged  at $p$.
\end{lemma}

 The first proof of this lemma appeared in \cite{rot76}, with a different formulation. The proof we present here is rather different.

\begin{proof}
The proof follows exactly the same lines as the one of Lemma~\ref{ccan}, replacing $\phi$ by $\widetilde{\phi}$, and $\varphi$ by $\widetilde{\varphi}  = \exp(\sum_j z_j Y_j)$. We are left to compute the coordinates $\widetilde{b_i}(z)$, $i=1, \dots,n$, of the vector $\widetilde{W}(z)=(\widetilde{\varphi}^{-1})_* X (p)$ in  the
basis $(Y_1(p), \dots, Y_n(p))$ of $T_pM$. It  results directly from Proposition~\ref{th:quasich} that
\begin{eqnarray*}
\widetilde{W}(z)  & = & \sum_{k = 0}^N \Frac{1}{k!} \left(\ad \sum_{i=1}^\ell z_j Y_j \right)^{k}X (p) + O(|z|^{N+1}),\\
& = &  \sum_{k=0}^N \sum_{l_1+\dots +l_{n} = k} a_l z_1^{l_1} \cdots
z_n^{l_n} Z_l(p) + O(|z|^{N+1}),
\end{eqnarray*}
where $Z_l$ belongs to $\Delta^w(p)$,
with $w=1+l_1w_1 +\cdots+l_nw_n$. Thus every coordinate $\widetilde{b_i}(z)$ of $\widetilde{W}(z)$ satisfies
\begin{eqnarray*}
\widetilde{b_i}(z)  & = & \sum_{k=0}^N \sum_{l_1+\dots +l_{n} = k} a_l z_1^{l_1} \cdots
z_n^{l_n} \widetilde{\beta}_i^l + O(|z|^{N+1}),
\end{eqnarray*}
 $\widetilde{\beta}_i^l$  being the
$i$th coordinate of $Z_l$ in the basis $(Y_1(p), \dots, Y_n(p))$.
This expression is similar to~\eqref{eq:bi}, and the same conclusion follows.
\end{proof}


\subsection{Algebraic coordinates}
\label{se:coordalg}

Let us recall the construction of the algebraic coordinates $(z_1, \dots, z_n)$ given in page~\pageref{pa:coordalg}. Let $Y_1, \dots, Y_n$ be an adapted frame  at $p$, and $(y_1,\dots,y_n)$ be local coordinates  centered at $p$ such
  that $\partial_{y_i}|_p = Y_i(p)$.  For $j=1, \dots,n$, we set
  \begin{equation}
    \label{eq:coordalg}
z_j = y_j - \sum_{k=2}^{w_j-1} h_k(y_1, \dots, y_{j-1}),
  \end{equation}
  where, for $k=2, \dots, w_j-1$,
  $$
h_k(y_1, \dots, y_{j-1})=
\sum_{\stackrel{\scriptstyle{|\alpha|=k}}{w(\alpha)<w_j}} \!\!\!\!
Y_1^{\alpha_1} \ldots Y_{j-1}^{\alpha_{j-1}} \Big(
y_j -\sum_{q=2}^{k-1} h_q(y)\Big)(p) \
\frac{y_1^{\alpha_1}}{\alpha_1!} \cdots
\frac{y_{j-1}^{\alpha_{j-1}}}{\alpha_{j-1}!},
  $$
  with $|\alpha|=\alpha_1 + \cdots + \alpha_n$.

\begin{lemma}
\label{le:coordalg}
The algebraic coordinates $(z_1, \dots, z_n)$ are privileged at $p$.
\end{lemma}

The proof of the lemma is based on the following result.

\begin{lemma}
\label{le:cnsorder}
A function $f$ is of order $\geq s$ at $p$ if and only if
$$
(Y_1^{\alpha_1} \cdots Y_n^{\alpha_n}f)(p)=0
$$
for all $\alpha$ such that $w(\alpha) <s$.
\end{lemma}

\begin{proof}
Let $f$ be a function of order $\geq s$ at $p$. Using the rules~\eqref{eq:ordcl}, we have $\mathrm{ord}_p(Y_i)\geq -w_i$ for $i=1,\dots,n$, and hence  $\mathrm{ord}_p(Y_1^{\alpha_1} \cdots Y_n^{\alpha_n}) > -s$ for every $\alpha=(\alpha_1,\dots,\alpha_n)$ such that $w(\alpha) <s$. Consequently, for such an $\alpha$ the function $Y_1^{\alpha_1} \cdots Y_n^{\alpha_n}f$ is of positive order, and so vanishes at $p$.\medskip

Conversely, let $f$ be a function of order $< s$ at $p$. We introduce the canonical coordinates of the second kind $(x_1,\dots,x_n)$ defined by means of the adapted basis $Y_1,\dots,Y_n$. Proposition~\ref{le:algorder} implies that there exists $\alpha$ such that $w(\alpha)=\mathrm{ord}_p (f) <s$ and
$(\partial_{x_1}^{\alpha_1} \cdots \partial_{x_n}^{\alpha_n}f)(p)\neq 0$.
Moreover, every vector field $Y_i$, $i=1,\dots,n$, writes in coordinates $x$ as
$$
 \sum_{j=1}^n Y_i^j (x) \partial_{x_j}, \quad \hbox{where } \mathrm{ord}_p (Y_i^j) \geq w_j-w_i.
$$
There also holds $Y_i^j (0) = \delta_{ij}$ since $Y_i(p) = \partial_{x_i}$. As a consequence,
$$
Y_1^{\alpha_1} \cdots Y_n^{\alpha_n} (p) = \partial_{x_1}^{\alpha_1} \cdots \partial_{x_n}^{\alpha_n}(p) + \sum_{w(\beta)<w(\alpha)} a_\beta \partial_{x_1}^{\beta_1} \cdots \partial_{x_n}^{\beta_n}(p),
$$
and thus $(Y_1^{\alpha_1} \cdots Y_n^{\alpha_n}f)(p) = (\partial_{x_1}^{\alpha_1} \cdots \partial_{x_n}^{\alpha_n}f)(p)\neq 0$ since $w(\alpha)=\mathrm{ord}_p (f)$. This ends the proof.
\end{proof}

\begin{proof}[Proof of Lemma~\ref{le:coordalg}]
Let $i \in \{1,\dots,n\}$. Note first that $Y_iz_i(p)=Y_iy_i(p)=1$,
which implies  $\mathrm{ord}_p (z_i) \leq w_i$. It remains to show
that $\mathrm{ord}_p (z_i) \geq w_i$. For this we will use the
criterion of Lemma~\ref{le:cnsorder}.

Let $\alpha$ such that $w(\alpha) <w_i$ (and so $|\alpha|<w_i$). Using
the expression~\eqref{eq:coordalg} of $z_i$, we obtain
\begin{multline}
\label{eq:yalpha}
    Y_1^{\alpha_1} \cdots Y_n^{\alpha_n} z_i = Y_1^{\alpha_1} \cdots
    Y_n^{\alpha_n} \left(y_i - \sum_{k=2}^{w_i-1} h_k(y)\right) \\
    = Y_1^{\alpha_1} \cdots Y_n^{\alpha_n} \left(y_i -
      \sum_{k=2}^{|\alpha|-1} h_k(y)\right) -  Y_1^{\alpha_1} \cdots
    Y_n^{\alpha_n} \left(\sum_{k=|\alpha|}^{w_i-1} h_k(y)\right).
\end{multline}
The functions $h_k$ are given by
$$
h_k(y)=
\sum_{\stackrel{\scriptstyle{|\beta|=k}}{w(\beta)<w_i}} \!\!\!\!
Y_1^{\beta_1} \ldots Y_{i-1}^{\beta_{i-1}} \Big(
y_i- \sum_{q=2}^{k-1} h_q(y)\Big)(p) \
\frac{y_1^{\beta_1}}{\beta_1!} \cdots
\frac{y_{i-1}^{\beta_{i-1}}}{\beta_{i-1}!}.
$$
Therefore, we clearly have $\left(Y_1^{\alpha_1} \cdots Y_n^{\alpha_n}
  h_k\right)(p)=0$
if $k>|\alpha|$, and
$$
\left(Y_1^{\alpha_1} \cdots Y_n^{\alpha_n}  h_{|\alpha|} \right)(p)=
Y_1^{\alpha_1} \cdots Y_n^{\alpha_n} \left(y_i -
  \sum_{k=2}^{|\alpha|-1} h_k(y)\right) (p).
$$
Plugging this expression into~\eqref{eq:yalpha}, we obtain
$(Y_1^{\alpha_1} \cdots Y_n^{\alpha_n} z_i)(p)=0$, which ends the
proof.
\end{proof}

\addcontentsline{toc}{section}{
    \textbf{References}}


\end{document}